\documentclass[final]{siamart220329}

\usepackage{showkeys}

\usepackage[show]{ed}

\usepackage{tikz}
\usepackage{pgf}
\usetikzlibrary{arrows}
\usetikzlibrary{calc}
\usetikzlibrary{decorations.pathmorphing}
\usepackage{xfrac}    

\usepackage{faktor}
\usepackage{soul}
\setstcolor{red}
\setulcolor{red}

\usepackage{rotating} 

\usepackage{mathrsfs}
\DeclareMathAlphabet{\mathpzc}{OT1}{pzc}{m}{it}

\usepackage{color}
\usepackage{amsmath}
\usepackage{graphicx}
\usepackage{graphics}
\usepackage{float}
\usepackage{amsfonts}
\usepackage{amssymb}%
\usepackage{latexsym}
\usepackage{psfrag}
\usepackage{subfigure}
\usepackage{accents}

\newtheorem{assumption}[theorem]{Assumption}
\newtheorem{remark}[theorem]{Remark}
\newtheorem{example}[theorem]{Example}

\newcommand{\cA}{{\cal A}}

\newcommand{\cH}{{\cal H}}

\newcommand{\cL}{{\cal L}}
\newcommand{\cP}{{\cal P}}

\newcommand{\cR}{{\cal R}}

\newcommand{\cZ}{{\cal Z}}
\newcommand{\bx}{x}

\newcommand{\ba}{a}

    \newcommand\quotient[2]{
        \mathchoice
            {
                \text{\raise1ex\hbox{$#1$}\Big/\lower1ex\hbox{$#2$}}%
            }
            {
                #1\,/\,#2
            }
            {
                #1\,/\,#2
            }
            {
                #1\,/\,#2
            }
    }

\newcommand{\re}{{\rm e}}
\newcommand{\ri}{{\rm i}}

\newcommand{\beq}{\begin{equation}}
\newcommand{\eeq}{\end{equation}}
\newcommand{\beqs}{\begin{equation*}}
\newcommand{\eeqs}{\end{equation*}}
\newcommand{\bit}{\begin{itemize}}
\newcommand{\eit}{\end{itemize}}
\newcommand{\ben}{\begin{enumerate}}
\newcommand{\een}{\end{enumerate}}
\newcommand{\bal}{\begin{align}}
\newcommand{\eal}{\end{align}}
\newcommand{\bals}{\begin{align*}}
\newcommand{\eals}{\end{align*}}
\newcommand{\bse}{\begin{subequations}}
\newcommand{\ese}{\end{subequations}}
\newcommand{\bpr}{\begin{proposition}}
\newcommand{\epr}{\end{proposition}}
\newcommand{\bre}{\begin{remark}}
\newcommand{\ere}{\end{remark}}
\newcommand{\bpf}{\begin{proof}}
\newcommand{\epf}{\end{proof}}
\newcommand{\ble}{\begin{lemma}}
\newcommand{\ele}{\end{lemma}}
\newcommand{\bco}{\begin{corollary}}
\newcommand{\eco}{\end{corollary}}
\newcommand{\bex}{\begin{example}}
\newcommand{\eex}{\end{example}}
\newcommand{\bth}{\begin{theorem}}
\newcommand{\enth}{\end{theorem}}

\newcommand{\Rea}{\mathbb{R}}
\newcommand{\Com}{\mathbb{C}}

\newcommand{\GR}{{\partial B_R}}

\newcommand{\pdiff}[2]{\frac{\partial #1}{\partial #2}}

\newcommand{\tendi}{\rightarrow \infty}

\def\XXint#1#2#3{{\setbox0=\hbox{$#1{#2#3}{\int}$}
     \vcenter{\hbox{$#2#3$}}\kern-.5\wd0}}

\newcommand*{\N}[1]{\left\|#1\right\|}

\allowdisplaybreaks[4]

\newcommand{\tfa}{\text{ for all }}
\newcommand{\tfor}{\text{ for }}

\newcommand{\tin}{\text{ in }}
\newcommand{\ton}{\text{ on }}
\newcommand{\tas}{\text{ as }}
\newcommand{\tand}{\text{ and }}
\newcommand{\tst}{\text{ such that }}

\newcommand{\vertiii}[1]{{\left\vert\kern-0.25ex\left\vert\kern-0.25ex\left\vert #1
    \right\vert\kern-0.25ex\right\vert\kern-0.25ex\right\vert}}

\newcommand{\DtN}{{\rm DtN}_k}

\definecolor{jwcol}{RGB}{27, 137, 18}  

\definecolor{dalcol}{rgb}{0.8,0,0}

\definecolor{escol}{rgb}{0,0,0.8}
\definecolor{estcol}{rgb}{0,0.5,0}
\definecolor{esnewcol}{rgb}{0,0.5,0}

\newcommand{\vin}{{v_{\rm in}}}
\newcommand{\vout}{{v_{\rm out}}}

\newcommand{\supp}{{\rm supp}}

\newcommand{\Csol}{C_{\rm sol}}

\newcommand{\Ccont}{{C_{\rm cont}}}

\newcommand{\tr}{{\rm tr}}

\newcommand{\Amin}{A_-}
\newcommand{\Amax}{A_+}
\newcommand{\cmin}{c_-}

\newcommand{\e}{\epsilon}

\newcommand{\domaingen}{\Omega}

\newcommand{\Rscat}{R_{\rm scat}}

\newcommand{\RPMLo}{R_{\rm PML, -}}
\newcommand{\RPMLt}{R_{\rm PML, +}}

\newcommand{\Ascatout}{A_{\rm out}}
\newcommand{\cscatout}{c_{\rm out}}
\newcommand{\Ascatin}{A_{\rm in}}
\newcommand{\cscatin}{c_{\rm in}}
\newcommand{\Omegain}{\Omega_{\rm in}}
\newcommand{\Omegaout}{\Omega_{\rm out}}
\newcommand{\uout}{u_{\rm out}}
\newcommand{\uin}{u_{\rm in}}

\newcommand{\loc}{\operatorname{loc}}
\newcommand{\Rtr}{R_{\tr} }

\newcommand{\fdspace}{\cH_h}

\newcommand{\euanspace}{, \,}

\newcommand{\CGo}{C_{\rm G1}}
\newcommand{\CGt}{C_{\rm G2}}
\newcommand{\Cell}{C_{\rm ell,1}}
\newcommand{\Celladj}{C_{\rm ell,2}}
\newcommand{\Coscil}{C_{\rm osc}}
\newcommand{\Ot}{\Omega_{\rm p}}
\newcommand{\OI}{\Omega_-}
\newcommand{\Gt}{\Gamma_{\rm p}}
\newcommand{\GI}{\Gamma_-}

\newcommand{\Zspace}{\cZ}

\makeatletter
\newcommand{\settheoremtag}[1]{
  \let\oldthetheorem\thetheorem
  \renewcommand{\thetheorem}{#1}
  \g@addto@macro\endtheorem{
    \addtocounter{theorem}{-1}
    \global\let\thetheorem\oldthetheorem}
  }
\makeatother

\usetikzlibrary{positioning}
\usepackage{lscape}

\definecolor{jeffColor}{RGB}{102, 0, 204}

\title{
Sharp preasymptotic error bounds for the Helmholtz \lowercase{$h$}-FEM
}

\author{
J.~Galkowski\thanks{Department of Mathematics, University College London, 25 Gordon Street, London, WC1H 0AY, UK,   \tt J.Galkowski@ucl.ac.uk}
\and
E.~A.~Spence\thanks{Department of Mathematical Sciences, University of Bath, Bath, BA2 7AY, UK, \tt E.A.Spence@bath.ac.uk }
}
\date{\today}

\begin{document}
\pagenumbering{arabic}

\maketitle

\begin{abstract}
In the analysis of the $h$-version of the finite-element method (FEM), with fixed polynomial degree $p$, applied to the Helmholtz equation with wavenumber $k\gg 1$, the \emph{asymptotic regime} is when
$(hk)^p \Csol$ is sufficiently small and the sequence of Galerkin solutions are quasioptimal; here $\Csol$ is the 
$L^2\to L^2$ norm of the Helmholtz solution operator, with $\Csol \sim k$ for nontrapping problems.
In the \emph{preasymptotic regime}, one expects that if $(hk)^{2p}\Csol$ is sufficiently small, then 
(for physical data) the relative error of the Galerkin solution is controllably small.

In this paper, we prove the natural error bounds in the preasymptotic regime for the variable-coefficient Helmholtz equation in the exterior of a Dirichlet, or Neumann, or penetrable obstacle (or combinations of these) 
and with the radiation condition \emph{either} realised exactly using the Dirichlet-to-Neumann map on the boundary of a ball \emph{or} approximated either by a radial perfectly-matched layer (PML) or an impedance boundary condition.
Previously, such bounds for $p>1$ were only available for Dirichlet obstacles with the radiation condition approximated by an impedance boundary condition. 
Our result is obtained via a novel generalisation of the ``elliptic-projection'' argument (the argument used to obtain the result for $p=1$) which can be applied to a wide variety of abstract Helmholtz-type problems.
\end{abstract}

\begin{AMS}
35J05, 65N15, 65N30, 78A45
\end{AMS}

\begin{keywords}
Helmholtz, FEM, high order, pollution effect, preasymptotic, perfectly-matched layer, elliptic projection.
\end{keywords}

\section{Introduction}

\subsection{Informal statement of the main result}

We consider the $h$-version of the finite-element method ($h$-FEM), where accuracy is increased by decreasing the meshwidth $h$ while keeping the polynomial degree $p$ constant, applied to the Helmholtz equation.

\begin{theorem}[Informal statement of the main result]\label{thm:informal}
Let $u$ be the solution to the variable-coefficient Helmholtz equation, with wavenumber $k>0$, in the exterior of a Dirichlet, or Neumann, or penetrable obstacle (or combinations of these) and with the radiation condition 
\emph{either} realised exactly using the Dirichlet-to-Neumann map on the boundary of a ball \emph{or} approximated either by a radial perfectly-matched layer (PML) or an impedance boundary condition. Let $\Csol$ be the $L^2\to L^2$ norm of the solution operator, with $\Csol\sim k$ for nontrapping problems.

Under the natural regularity assumptions on the domain and coefficients for using degree $p$ polynomials, if 
\beq\label{eq:threshold}
(hk)^{2p}\Csol
\text{ is sufficiently small} 
\eeq
then the Galerkin solution $u_h$ exists, is unique, and satisfies
\begin{gather}\label{eq:H1bound}
\N{u-u_h}_{H^1_k(\Omega)}\leq C \Big(1 + (hk)^p \Csol\Big) \min_{v_h \in \fdspace} \N{u-v_h}_{H^1_k(\Omega)},\\
\label{eq:L2bound}
\N{u-u_h}_{L^2(\Omega)}\leq C \Big( hk + (hk)^p \Csol\Big) \min_{v_h \in \fdspace} \N{u-v_h}_{H^1_k(\Omega)}.
\end{gather}
Furthermore, if the data is $k$-oscillatory and sufficiently regular (in a sense made precise below), then
\beq\label{eq:rel_error}
\frac{\N{u-u_h}_{H^1_k(\Omega)}}
{
\N{u}_{H^1_k(\Omega)}
}
\leq C \Big(1  + (hk)^p \Csol\Big)(hk)^p;
\eeq
i.e., the relative $H^1_k$ error can be made controllably small by making $(hk)^{2p}\Csol$ sufficiently small.
\end{theorem}
The norm $\|\cdot\|_{H^1_k(\Omega)}$ in the bounds above is defined by
\beq\label{eq:1knorm}
\N{v}^2_{H^1_k(\domaingen)}:= k^{-2}\N{\nabla v}^2_{L^2(\domaingen)} + \N{v}^2_{L^2(\domaingen)}.
\eeq

\subsection{The context and novelty of the main result}\label{sec:context}

The fact that, for oscillatory data, the relative $H^1_k$ error for the Helmholtz $h$-FEM is controllably small if $(hk)^{2p}\Csol$ is sufficiently small was famously identified for 1-d nontrapping problems by the work of Ihlenburg and Babu\v{s}ka \cite{IhBa:95a, IhBa:97} (see \cite[Page 350, penultimate displayed equation]{IhBa:97}, \cite[Equation 4.7.41]{Ih:98}).  
The bounds \eqref{eq:H1bound} and \eqref{eq:L2bound} have previously been obtained for
\ben
\item $p=1$, 
for Helmholtz problems with either an impedance boundary condition \cite[Theorem 6.1]{Wu:14}, \cite[Theorem 2]{BaChGo:17}, 
or truncation via the exact Dirichlet-to-Neumann map \cite[Theorem 4.1]{LSW2}, or
truncation via a radial, $k$-independent PML \cite[Theorem 4.4]{LiWu:19}, or truncation via a radial, $k$-dependent PML \cite[Theorem 7.2]{ChGaNiTo:22},
\item $p\in \mathbb{Z}^+$, the constant-coefficient Helmholtz equation with no obstacle and an impedance boundary condition approximating the radiation condition \cite[Theorem 5.1]{DuWu:15},
\item $p\in \mathbb{Z}^+$, the variable-coefficient Helmholtz equation in the exterior of a Dirichlet obstacle with an impedance boundary condition approximating the radiation condition \cite[Theorem 2.39]{Pe:20}.
\een
The bounds in Point 1 for $p=1$ come from the so-called \emph{elliptic projection} argument, which proves error bounds under the condition ``$(hk)^{p+1}\Csol$ is sufficiently small''; i.e., the sharp condition when $p=1$, but not when $p>1$. 
The initial ideas behind this argument were introduced in the Helmholtz context  in \cite{FeWu:09, FeWu:11} for interior-penalty discontinuous Galerkin methods, and 
then further developed for the standard FEM and continuous interior-penalty methods in \cite{Wu:14,ZhWu:13}. 

The present paper proves the bounds \eqref{eq:H1bound}, \eqref{eq:L2bound}, and \eqref{eq:rel_error} for the $h$-FEM assuming only that the sesquilinear form is continuous, satisfies a G\aa rding inequality, and satisfies certain standard elliptic-regularity assumptions, therefore covering a variety of scatterers and methods for truncating the exterior domain. 
Regarding the latter:~in this paper we consider truncating with the exact Dirichlet-to-Neumann map on the boundary of a ball, with a radial PML, or with an impedance boundary condition.

Since the preprint of this paper appeared, its ideas have been used in 
\cite{LiWu:23, ChSp:24, ChGaSp:24}, with \cite{LiWu:23} analysing high-order continuous interior-penalty methods for the Helmholtz PML problem (generalising the $p=1$ results in \cite{LiWu:19}), \cite{ChSp:24} 
analysing the geometric error for the Helmholtz $h$-FEM, and \cite{ChGaSp:24} obtaining preasymptotic error bounds for the Maxwell $h$-FEM.

\subsection{Statement of the main result in abstract form}\label{sec:statement}

Let $\mathcal{H}\subset \mathcal{H}_0\subset \mathcal{H}^*$ be Hilbert spaces, with $\cH^*$ the space of anti-linear functionals on $\cH$, 
$\cH_0$ identified with its dual, 
and
 $\mathcal{H}\subset \mathcal{H}_0$ compact. Let  $a:\mathcal{H}\times \mathcal{H}\to \mathbb{C}$ be a 
 sesquilinear form; i.e., $a$ is linear in its first argument, anti-linear in its second argument.
We assume that $a$ is continuous, i.e., 
\begin{equation}
\label{eq:contAb}
|a(u,v)|\leq \Ccont \N{u}_{\mathcal{H}}\N{v}_{\mathcal{H}}
\quad\tfa u,v\in \cH,
\end{equation}
and satisfies the G\aa rding inequality 
\begin{equation}
\label{eq:Garding}
\Re a(v,v)\geq \CGo\N{v}_{\mathcal{H}}^2-\CGt\N{v}^2_{\mathcal{H}_0} \quad \tfa v \in \cH,
\end{equation}
for some $\Ccont, \CGo,\CGt>0$.

\begin{assumption}[Abstract elliptic regularity bounds for $a$]\label{ass:er}
Let $\Zspace_0=\mathcal{H}_0$, $\Zspace_1=\mathcal{H}$, and $\Zspace_j\subset \Zspace_{j-1}$ for $j=2,\dots,\ell+1$ be such that $\Zspace_j$ is dense in $\Zspace_{j-1}$. There exists a $\Cell>0$ such that, for all $u\in \mathcal{H}$, 
\begin{equation}
\label{eq:ellipticRegA}
\N{u}_{\Zspace_j}\leq \Cell\Big(\N{u}_{\mathcal{H}_0}+\sup_{\substack{v\in \mathcal{H}\euanspace \|v\|_{(\Zspace_{j-2})^*}=1}}\big| a(u,v)\big|\,\Big),\quad
j=2,\dots,\ell+1,
\end{equation}
with $u\in \Zspace_j$ if the right-hand side of \eqref{eq:ellipticRegA} is finite.
In addition,
for all $u\in \mathcal{H}$, 
\begin{equation}
\label{eq:ellipticReg}
\|u\|_{\Zspace_j}\leq \Cell\Big(\N{u}_{\mathcal{H}_0}+\sup_{\substack{v\in \mathcal{H}\euanspace \|v\|_{(\Zspace_{j-2})^*}=1}}\big|(\Re a)(u,v)\big|\,\Big),
\quad j=2,\dots,\ell+1,
\end{equation}
with $u\in \Zspace_j$ if the right-hand side of \eqref{eq:ellipticReg} is finite, where the sesquilinear form $\Re a$ is defined by 
\beq\label{eq:reala}
(\Re a)(u,v):=\tfrac{1}{2}\big(\,a(u,v)+\overline{a(v,u)}\,\big).
\eeq
\end{assumption}

Recalling the one-to-one correspondence between sesquilinear forms $a:\cH\times\cH\to \mathbb{C}$ and operators $\cA: \cH\to \cH^*$ given by 
$a(u,v) = \langle \cA u,v\rangle_{\cH^*\times \cH}$ (see, e.g., \cite[Page 42]{Mc:00}), 
we see that $\Re a$ \eqref{eq:reala} is the sesquilinear form corresponding to the operator $\Re \cA: =(\cA + \cA^*)/2$.

\begin{remark}\label{rem:omega}
Note that $\Re a$ in~\eqref{eq:Garding} and~\eqref{eq:reala} could be replaced by 
$\Re (\re^{\ri\omega}a)$, so long as one uses the same value of $\omega$ in both conditions. 
Remark \ref{rem:JeffPML} below describes a situation where this is useful.
\end{remark}

\begin{example}\label{ex:PML}
For the Helmholtz equation outside a Dirichlet obstacle with radial PML truncation and $\Omega$ the truncated exterior domain,
$\cH_0=L^2(\Omega)$, $\cH= H^1_0(\Omega)$, and $\Zspace_j = H^j(\Omega) \cap H^1_0(\Omega)$.
 Assumption \ref{ass:er} is then elliptic regularity for the Helmholtz PML operator and its real part, which both hold if the coefficients of the Helmholtz equation are in $C^{\ell-1,1}$, the PML scaling function is $C^{\ell, 1}$, 
$\partial \Omega$ is $C^{\ell,1}$ 
 and one works with the Sobolev norms where each derivative is scaled by $k^{-1}$ (see Lemma \ref{lem:checkPML} below).
\end{example}

Given $g\in \mathcal{H}^*$, {suppose that $u\in \cH$ satisfies
\begin{equation}
\label{eq:vp_abs}
a(u,v)=\langle g,v\rangle\qquad \text{ for all }v\in\mathcal{H},
\end{equation}
where $\langle\cdot,\cdot\rangle$ is the duality pairing between $\cH^*$ and $\cH$. 
Since $a$ is continuous \eqref{eq:contAb} and satisfies the G\aa rding inequality \eqref{eq:Garding} with $\cH \subset \cH_0$ compact, 
uniqueness of the solution to \eqref{eq:vp_abs} is equivalent to existence; see, e.g., \cite[Theorem 2.32]{Mc:00}.

Given a sequence of finite dimensional subspace $\{\mathcal{H}_h\}_{h>0}$ with $\cH_h \subset \mathcal{H}$, the Galerkin method seeks approximations of $u$, $\{u_h\}_{h>0}$ with $u_h\in\cH_h$, such that
\begin{equation}
\label{eq:Galerkin_abs}
a(u_h,v_h)=\langle g,v_h\rangle\,\,\text{ for all }\,\,v_h\in\mathcal{H}_h.
\end{equation}

\begin{theorem}[Abstract generalisation of the elliptic-projection argument]\label{thm:ep_abs}

Let $a:\mathcal{H}\times \mathcal{H}\to \mathbb{C}$  satisfy~\eqref{eq:contAb},~\eqref{eq:Garding}, and Assumption \ref{ass:er} for some $\ell \in\mathbb{Z}^+$, and suppose that 
\eqref{eq:vp_abs} has a unique solution. 
Define $\mathcal{R}^*:\mathcal{H}^*\to \mathcal{H}$  by 
\beq\label{eq:R*_abs}
a(w,\mathcal{R}^*v)=\langle w,v\rangle\qquad \text{ for all } w\in \mathcal{H},\,v\in\mathcal{H}^*,
\eeq
and
\beq\label{eq:eta_def_abs}
\eta(\cH_h):=\|(I-\Pi_h)\mathcal{R}^*\|_{\cH_0 \to \cH},
\eeq
where $\Pi_h:\mathcal{H}\to \mathcal{H}_h$ is the orthogonal projection. Then there exist $C_1, C_2, C_3>0$
such that 
 if $h$ satisfies
\beq\label{eq:threshold_abs}
\eta(\cH_h)\|I-\Pi_h\|_{\Zspace_{\ell+1}\to \mathcal{H}} \leq C_1,
\eeq
then the solution $u_h$ to~\eqref{eq:Galerkin_abs} exists, is unique, and satisfies
\begin{align}\label{eq:H1bound_abs}
\N{u-u_h}_{\mathcal{H}} &\leq C_2\big(1+\eta(\cH_h)\big) \N{(I-\Pi_h)u}_{\cH},\\
\label{eq:L2bound_abs}
\N{u-u_h}_{\mathcal{H}_0} &\leq C_3\, \eta(\cH_h) \N{(I-\Pi_h)u}_{\cH},
\end{align}
In addition, for all $\Coscil>0$ there exists $C_4>0$ such that if 
\beq\label{eq:oscil}
\N{g}_{\Zspace_{\ell-1}}\leq \Coscil \N{g}_{\cH^*}
\eeq
and $h$ satisfies \eqref{eq:threshold_abs} then
\beq\label{eq:rel_error_abs}
\frac{\N{u-u_h}_{\mathcal{H}} 
}{
\N{u}_{\cH}
}
\leq C_4 \big(1+\eta(\cH_h)\big)\N{I-\Pi_h}_{\Zspace_{\ell+1}\to \cH};
\eeq
i.e., the relative error in $\cH$ can be made controllably small by making $\eta(\cH_h)\N{I-\Pi_h}_{\Zspace_{\ell+1}\to \cH}$ sufficiently small.
\end{theorem}

By the order of quantifiers in Theorem \ref{thm:ep_abs}, $C_1, C_2,$ and $C_3$ depend only on $\Ccont, \CGo, \CGt, \Cell$, and $\ell$, and $C_4$ depends only on $\Ccont, \CGo, \CGt, \Cell$, $\ell$, and $\Coscil$.

The bound \eqref{eq:H1bound_abs} implies 
 the result that the sequence of Galerkin solutions are quasioptimal 
 if $\eta(\cH_h)$ is sufficiently small -- with this the so-called \emph{asymptotic regime}.
Theorem \ref{thm:ep_abs}, however, holds under the less-restrictive condition that $\eta(\cH_h)\|I-\Pi_h\|_{\Zspace_{\ell+1}\to \mathcal{H}}$ be small (i.e., the condition \eqref{eq:threshold_abs}), 
and so is valid in part of the preasymptotic regime.

The bounds \eqref{eq:H1bound_abs}, \eqref{eq:L2bound_abs}, and \eqref{eq:rel_error_abs} and the meshthreshold \eqref{eq:threshold_abs} 
all involve the quantity $\eta(\cH_h)$, which measures how well solutions of the adjoint problem are approximated in the space $\cH_h$. 
Bounds on $\eta(\cH_h)$ are given in \cite{MeSa:10, MeSa:11, EsMe:12, ChNi:20, LSW3, LSW4, GLSW1, BeChMe:22}. 
The following bound on $\eta(\cH_h)$ 
is essentially the one in \cite{ChNi:20}, although our proof is different. We include this bound here 
both for completeness, and 
because, after establishing the key intermediate result for proving Theorem \ref{thm:ep_abs} (Lemma \ref{lem:abs1} below) our proof of 
the bound on $\eta(\cH_h)$ is very short (see \S\ref{sec:eta_abs} below).

\begin{assumption}[Elliptic regularity assumption for the adjoint sesquilinear form]\label{ass:er2}
With $\Zspace_j$, $j=0,\ldots,\ell+1$, as in Assumption \ref{ass:er}, there exists a $\Celladj>0$ such 
\begin{equation}
\label{eq:ellipticRegAadj}
\N{v}_{\Zspace_j}\leq \Celladj\Big(\N{v}_{\mathcal{H}_0}+\sup_{\substack{u\in \mathcal{H}\euanspace \|u\|_{(\Zspace_{j-2})^*}=1}}| a(u,v)|\Big),\quad
j=2,\dots,\ell+1,
\end{equation}
with $v\in \Zspace_j$ if the right-hand side of \eqref{eq:ellipticRegAadj} is finite.
\end{assumption}

\begin{theorem}[Bound on $\eta(\cH_h)$]\label{thm:eta_abs}
Suppose that the assumptions of Theorem \ref{thm:ep_abs} hold and, additionally, Assumption \ref{ass:er2} holds. 
With $\cR^*$ and $\eta(\cH_h)$ defined by \eqref{eq:R*_abs} and \eqref{eq:eta_def_abs} respectively, 
there exists $C>0$ 
such that
\beq\label{eq:eta_bound_abs}
\eta(\cH_h) \leq C\Big(
\N{(I-\Pi_h)}_{\Zspace_{2}\to \cH}+ 
\N{(I-\Pi_h)}_{\Zspace_{\ell+1}\to \cH}\N{\cR^*}_{\cH_0\to \cH_0}
\Big).
\eeq
\end{theorem}

\begin{example}\label{ex:threshold}
In \S\ref{sec:PML} and \S\ref{sec:impedance} below we show how Helmholtz problems with the radiation condition 
\emph{either} realised by the exact Dirichlet-to-Neumann map on the boundary of a ball \emph{or}
approximated by either a radial  PML or an impedance boundary condition, respectively, fit into the abstract framework of Theorems \ref{thm:ep_abs} and \ref{thm:eta_abs}. In both these cases,
the norm of the adjoint solution operator, i.e., $\N{\cR^*}_{\cH_0\to \cH_0}$, is the same as the norm of the solution operator of the original (non-adjoint) problem, which we denote by $\Csol$. Furthermore, with $\{\cH_h\}_{h>0}$ corresponding to the standard finite-element spaces of piecewise degree-$p$ polynomials on shape-regular simplicial triangulations, indexed by the meshwidth $h$, 
\beq\label{eq:poly_approx}
\N{(I-\Pi_h)}_{\Zspace_{m+1}\to \cH} \leq C(hk)^m \quad \tfor 0\leq m\leq p.
\eeq
The meshthreshold \eqref{eq:threshold_abs} then becomes that $(hk)^{2\ell}\Csol$ is sufficiently small when $\ell \leq p$. Recall that $\ell$ is a parameter in the elliptic-regularity assumptions (Assumptions \ref{ass:er} and \ref{ass:er2}). If the polynomial degree $p$ is taken to be $\ell$ then \eqref{eq:threshold_abs} becomes \eqref{eq:threshold}; the bounds \eqref{eq:H1bound_abs} and \eqref{eq:L2bound_abs} then become \eqref{eq:H1bound} and \eqref{eq:L2bound}, respectively.
\end{example}

\section{Proofs of the main results (Theorems \ref{thm:ep_abs} and \ref{thm:eta_abs})}

\subsection{Construction of a regularizing operator that produces coercivity when added to $a$}

\begin{lemma}\label{lem:abs1}
Suppose that $a:\mathcal{H}\times \mathcal{H}\to \mathbb{C}$ satisfies~\eqref{eq:contAb}, \eqref{eq:Garding}, and Assumption \ref{ass:er} for some $\ell \in \mathbb{Z}^+$.
Then there exists $S:\mathcal{H}_0\to \mathcal{H}_0$ self adjoint and $c,C>0$ such that, with 
\begin{align}
\label{eq:tAab}
\widetilde{a}(u,v)&:=a(u,v)+\langle Su,v\rangle_{\mathcal{H}_0},\\
\label{eq:aTildeAb}
\Re\widetilde{a}(v,v) &\geq c\N{v}_{\mathcal{H}}^2 \,\,\tfa v \in \cH,\\
\label{eq:Sreg}
\|S\|_{\mathcal{H}_0\to \Zspace_j}&\leq C,\qquad j=0,\dots,\ell-1,
\end{align}
and $\widetilde{\mathcal{R}}:\mathcal{H}^*\to \mathcal{H}$ defined by 
\begin{gather}\label{eq:tildeR}
\widetilde{a}(\widetilde{\mathcal{R}}f,v)=\langle f,v\rangle\quad\tfa v\in \mathcal{H},\,f\in\mathcal{H}^*,
\end{gather}
is well defined with
\beq\label{eq:Rtildereg}
\|\widetilde{\mathcal{R}}\|_{\Zspace_{j-2}\to \Zspace_j}\leq C,\qquad 2\leq j\leq \ell+1.
\eeq
\end{lemma}

\bre[Relation with the original elliptic-projection argument]
The original elliptic-projection argument  \cite{FeWu:09, FeWu:11} uses coercivity of \eqref{eq:tAab} with $S$ a sufficiently large multiple of the identity (see \cite[Lemma 5.1]{FeWu:09}, \cite[Lemma 4.1]{FeWu:11}). This particular $S$ satisfies Lemma \ref{lem:abs1} with $\ell=1$; the threshold \eqref{eq:threshold_abs} in Theorem \ref{thm:ep_abs} is then ``$\eta(\cH_h)\|I-\Pi_h\|_{\Zspace_{2}\to \mathcal{H}}$ sufficiently small", which (by \eqref{eq:eta_bound_abs} and \eqref{eq:poly_approx}) becomes the condition ``$(hk)^{\ell+1}\Csol$ sufficiently small'' discussed in \S\ref{sec:context}.
\ere

The proof of Lemma \ref{lem:abs1} uses the spectral theorem for bounded self-adjoint operators, $B:\cH\to \cH^*$, which we recap here.
With $\cH_0$ and $\cH$ as in \S\ref{sec:statement}, let $b$ be a sesquilinear form on $\cH$ satisfying $b(u,v)=\overline{b(v,u)}$, with associated operator $B$; i.e., $b(u,v)= \langle Bu,v\rangle$ for all $u,v\in \cH$. If $b$ satisfies the G\aa rding inequality \eqref{eq:Garding} (with $a$ replaced by $b$) then 
there exist an orthonormal basis (in $\cH_0$) of eigenfunctions of $B$, $\{\phi_j\}_{j=1}^\infty$, with associated eigenvalues satisfying $\lambda_1\leq \lambda_2 \leq \ldots$ with $\lambda_j\to \infty$ as $j\to\infty$. Furthermore, for all $u \in \cH$, 
\beq\label{eq:FC0a}
Bu = \sum_{j=1}^\infty \lambda_j \langle \phi_j , u\rangle \phi_j
\eeq
(where the sum converges in $\cH^*$); see, e.g., \cite[Theorem 2.37]{Mc:00}.
Given a bounded function $f$, we define $f(B):\cH_0\to \cH_0$ by 
\beq
f(B)u:= \sum_{j=1}^\infty f(\lambda_j) \langle \phi_j , u\rangle \phi_j,
\label{eq:FC1}
\quad\text{ so that }\quad
\N{f(B)}_{\cH_0\to \cH_0} \leq \sup_{\lambda \in [\lambda_1,\infty) } |f(\lambda)|.
\eeq

\begin{proof}[Proof of Lemma \ref{lem:abs1}]
Let $\mathcal{P}:\cH\to \cH^*$ be the operator associated with the sesquilinear form $\Re a$ defined by \eqref{eq:reala}, i.e., $(\Re a)(u,v) = \langle \cP u, v\rangle$ for all $u,v\in \cH$; observe that $\cP$ is self-adjoint. Since $(\Re a)$ also satisfies the G\aa rding equality satisfied by $a$ \eqref{eq:Garding}, the spectral theorem recapped above applies. 
Let $\{\lambda_j\}_{j=1}^\infty$ be the eigenvalues of $\cP$ with $\lambda_{j}\leq \lambda_{j+1}$, $j=1,\dots$,
let $\psi\in C_{\rm comp}^\infty(\mathbb{R};[0,\infty))$ be such that
\beq\label{eq:psi}
x+\psi(x)\geq 1 \quad\tfor x\geq \lambda_1,
\eeq
and let 
$
S:=\psi(\mathcal{P}),
$
in the sense of \eqref{eq:FC1}.

We now use \eqref{eq:ellipticReg} to prove that $S:\mathcal{H}_0\to \Zspace_j$ satisfies \eqref{eq:Sreg}.
Since $\psi$ has compact support, the function $t\mapsto t^m \psi(t)$ is bounded for any $m\geq 0$. Thus \eqref{eq:FC1} implies that, for any $m\geq 0$,
\begin{equation}
\label{eq:Point1}
\N{\mathcal{P}^m \psi(\mathcal{P})}_{\cH_0\to \cH_0}\leq C_m.
\end{equation}
By \eqref{eq:ellipticReg},
$$
\N{\psi(\mathcal{P})}_{\cH_0\to \Zspace_j}\leq \Cell\Big(\N{\psi(\mathcal{P})}_{\cH_0\to \cH_0}+\N{\mathcal{P}\psi(\mathcal{P})}_{\cH_0\to \Zspace_{j-2}}\Big),\quad j=2,\dots,\ell+1,
$$
so that, by induction and \eqref{eq:Point1},
$$
\N{S}_{\cH_0\to \Zspace_{\ell-1 }}=
\N{\psi(\mathcal{P})}_{\cH_0\to \Zspace_{\ell-1 }}\leq C_\ell\sum_{j=0}^{
\lceil (\ell-1 )/2\rceil
}\N{\mathcal{P}^j\psi(\mathcal{P})}_{\cH_0\to \cH_0}\leq C_\ell.
$$

We now show that $\widetilde{a}$ satisfies \eqref{eq:aTildeAb}.
By the definitions of $\cP$ and $S$,
the eigenfunction expansions \eqref{eq:FC0a} and \eqref{eq:FC1}, and the inequality \eqref{eq:psi}, for all $v\in \cH$,
$$
\Re \widetilde{a}(v,v)= \Re a(v,v)+\langle \psi(\mathcal{P})v,v\rangle= \langle (\mathcal{P}+\psi(\mathcal{P}))v,v\rangle\geq \N{v}_{\mathcal{H}_0}^2.
$$
Since $\psi\geq 0$, $S$ is positive, and thus $\Re \widetilde{a}(v,v)\geq \Re a(v,v)$ for all $v \in \cH$, 
for any $\e>0$ and for all $v\in \cH$,
$$
\Re \widetilde{a}(v,v) \geq \e \Re a(v,v)+(1-\e)\Re \widetilde{a}(v,v)\geq \e \CGo\N{v}_{\mathcal{H}}^2-\CGt\e \N{v}_{\mathcal{H}_0}^2+(1-\e)\|v\|_{\mathcal{H}_0}^2.
$$
Therefore, if 
$\e=(1+\CGt)^{-1}$, then
$$
\Re \widetilde{a}(v,v)\geq \CGo (1 + \CGt)^{-1}\N{v}_{\mathcal{H}}^2;
$$
i.e., $\widetilde{a}$ is coercive. The existence of $\widetilde{\mathcal{R}}:\mathcal{H}^*\to \mathcal{H}$ satisfying \eqref{eq:tildeR}
and
$
\|\widetilde{\mathcal{R}}\|_{\mathcal{H}^*\to \mathcal{H}}\leq C
$
then follows from the Lax--Milgram lemma (see, e.g., \cite[Lemma 2.32]{Mc:00}).
Finally, to see that 
$$
\|\widetilde{\mathcal{R}}\|_{\Zspace_{j-2}\to \Zspace_j}\leq C,\qquad 2\leq j\leq \ell+1,
$$
observe that, since $S$ is self-adjoint and satisfies \eqref{eq:Sreg}, for $v\in (\Zspace_{j-2})^*$,
\begin{align*}
| a(\widetilde{\mathcal{R}}g,v)|=|  \widetilde{a}(\widetilde{\mathcal{R}}g,v)-\langle S\widetilde{\mathcal{R}}g,v\rangle|
&\leq | \widetilde{a}(\widetilde{\mathcal{R}}g,v)|+|\langle S\widetilde{\mathcal{R}}g,v\rangle|\\
&\hspace{-1cm}\leq | \langle v,g\rangle|+\|v\|_{(\Zspace_{j-2})^*}\|S\|_{\mathcal{H}\to \Zspace_{j-2}}\|(\widetilde{\mathcal{R}})^*\|_{\mathcal{H}^*\to \mathcal{H}}\|g\|_{\mathcal{H}^*}\\
&\hspace{-1cm}\leq \|v\|_{(\Zspace_{j-2})^*}(\|g\|_{\Zspace_{j-2}}+C\|g\|_{\mathcal{H}^*}),
\end{align*}
and the claim follows from \eqref{eq:ellipticRegA}.
\end{proof}

\subsection{Proof Theorem \ref{thm:ep_abs} using Lemma \ref{lem:abs1}}

By, e.g., \cite[Theorem 2.34]{Mc:00}, the operator associated to the sesquilinear form $a$ is Fredholm of index zero. 
The fact that the solution to \eqref{eq:vp_abs} is unique therefore implies that the solution exists, and also implies that $\cR^*$ is well-defined (by, e.g., \cite[Theorem 2.27]{Mc:00}).

We claim it is sufficient to prove the bounds \eqref{eq:H1bound_abs} and \eqref{eq:L2bound_abs} under the assumption of existence. Indeed, by uniqueness of the variational problem \eqref{eq:vp_abs}, either of the bounds \eqref{eq:H1bound_abs} or \eqref{eq:L2bound_abs} under the assumption of existence implies uniqueness of $u_h$, 
 and uniqueness implies existence for the finite-dimensional Galerkin linear system. 

We next show that the bound \eqref{eq:H1bound_abs} follows from \eqref{eq:L2bound_abs}. 
In the rest of the proof, $C$ denotes a constant (whose value may change from line to line) that only depends on $\Ccont,\CGo,\CGt,\Cell,$ and $\ell$. 
By the G\aa rding inequality \eqref{eq:Garding}, Galerkin orthogonality 
\beq\label{eq:Gog}
a(u-u_h, v_h)= 0 \quad\tfa v_h\in \fdspace,
\eeq
and \eqref{eq:L2bound_abs}, for any $v_h\in \cH_h$,
\begin{align}
\label{eq:L2toH1a}
\N{u-u_h}^2_{\cH} &\leq C\Big[ \big| a(u-u_h, u-v_h)\big| +   \N{u-u_h}^2_{\cH_0}\Big]\\
&\leq C\Big[ \N{u-u_h}_{\cH} \N{ u-v_h}_{\cH}
+  \Big(
\eta(\cH_h)
\N{(I-\Pi_h)u}_{\cH}
\Big)^2\Big].\label{eq:L2toH1}
\end{align}
The bound \eqref{eq:H1bound_abs} on the error in $\cH$ then follows by using the inequality $2ab\leq \epsilon a^2 + b^2/\epsilon$ for all $a,b,\epsilon>0$ in the first term on the right-hand side of \eqref{eq:L2toH1}, and then using the inequality $a^2+b^2\leq (a+b)^2$ for $a,b>0$.

We now prove \eqref{eq:L2bound_abs}. 
By the definition of $\cR^*$, Galerkin orthogonality \eqref{eq:Gog}, and the definition of $\widetilde a$~\eqref{eq:tAab}
\begin{align}\nonumber
\N{u-u_h}^2_{\mathcal{H}_0} & = a \big( u-u_h, \cR^*(u-u_h)  \big) = a \big(u-u_h, \cR^* (u-u_h) - v_h \big)\\
& = \widetilde{a} \big( u-u_h, \cR^* (u-u_h) - v_h \big) - \big\langle S(u-u_h), \cR^*(u-u_h) -v_h\big\rangle_{\mathcal{H}_0}.
\label{eq:JJS1Ab}
\end{align}
Let $\widetilde{\Pi}_h: \cH\to\cH_h$ 
be the solution of the variational problem
\beqs
\widetilde{a} (w_h, \widetilde{\Pi}_h v ) = \widetilde{a} (w_h , v) \quad\tfa w_h \in \cH_h.
\eeqs
Since $\widetilde{a}$ is continuous (by \eqref{eq:contAb} and \eqref{eq:Sreg}) and coercive 
(by \eqref{eq:aTildeAb}), by the Lax--Milgram lemma and C\'ea's lemma (see, e.g., \cite[Theorem 2.8.1]{BrSc:08}), 
$\widetilde{\Pi}_h$ is well-defined with
\beq\label{eq:epqoAb}
\big\|(I-\widetilde{\Pi}_h)v\big\|_{\cH}\leq C\N{(I-\Pi_h)v}_{\cH}. 
\eeq
The definition of $\widetilde{\Pi}_h$ implies the Galerkin orthogonality
\beq\label{eq:epGogAb}
\widetilde{a} \big(w_h, (I-\widetilde{\Pi}_h) u \big) =0\quad\tfa w_h \in \cH_h.
\eeq
We now choose $v_h= \widetilde{\Pi}_h \cR^* (u-u_h)$ in \eqref{eq:JJS1Ab} so that, by \eqref{eq:epGogAb}, 
\begin{align} \nonumber
&\N{u-u_h}^2_{\mathcal{H}_0} \\ \nonumber
& = \widetilde{a} \big((I-\Pi_h)u, (I-\widetilde{\Pi}_h)\cR^* (u-u_h) \big) - \big\langle u-u_h, S^*(I-\widetilde{\Pi}_h)\cR^* (u-u_h)\big\rangle_{\mathcal{H}_0}\\ 
&\leq C \N{(I-\Pi_h)u}_{\cH}\big\|(I-\widetilde{\Pi}_h)\cR^* (u-u_h)\big\|_{\cH} 
+ \N{u-u_h}_{\mathcal{H}_0} \big\|S^*(I-\widetilde{\Pi}_h)\cR^* (u-u_h)\big\|_{\mathcal{H}_0}.
\label{eq:JJS2Ab}
\end{align}
By \eqref{eq:epqoAb} 
and the definition of $\eta(\cH_h)$ \eqref{eq:eta_def_abs},
\beq\label{eq:JJS3Ab}
\big\|(I-\widetilde{\Pi}_h)\cR^* (u-u_h)\big\|_{\cH}\leq C
\N{(I-\Pi_h)\cR^* (u-u_h)}_{\cH}
\leq C \eta(\cH_h) \N{u-u_h}_{\mathcal{H}_0}.
\eeq
We now claim that the bound \eqref{eq:L2bound_abs} follows if we can prove that, for all $v\in \cH$,
\beq\label{eq:STP1Ab}
\big\|S^* (I- \widetilde{\Pi}_h) v \big\|_{\mathcal{H}_0} \leq C\|I-\Pi_h\|_{\Zspace_{\ell+1}\to \mathcal{H}} \big\|(I- \widetilde{\Pi}_h) v\big\|_{\cH}.
\eeq
Indeed, we use \eqref{eq:STP1Ab}, with $v=\cR^* (u-u_h)$, in the second term on the right-hand side of \eqref{eq:JJS2Ab} to obtain
\begin{align}\nonumber
\N{u-u_h}^2_{\mathcal{H}_0} 
&\leq C \N{(I-\Pi_h)u}_{\cH}\big\|(I-\widetilde{\Pi}_h)\cR^* (u-u_h)\big\|_{\cH} \\
&\qquad+ C\|I-\Pi_h\|_{\Zspace_{\ell+1}\to \mathcal{H}}\N{u-u_h}_{\mathcal{H}_0} \big\|(I-\widetilde{\Pi}_h)\cR^* (u-u_h)\big\|_{\mathcal{H}}.
\label{eq:JJS2Ab_new}
\end{align}
We then use \eqref{eq:JJS3Ab} in both terms on the right-hand side of \eqref{eq:JJS2Ab_new} to obtain 
\begin{align*}
\N{u-u_h}^2_{\mathcal{H}_0} \leq C \eta(\cH_h)\N{(I-\Pi_h)u}_{\cH}\N{u-u_h}_{\cH_0}+ C\eta(\cH_h) \|I-\Pi_h\|_{\Zspace_{\ell+1}\to \mathcal{H}}\N{u-u_h}_{\mathcal{H}_0}^2,
\end{align*}
from which \eqref{eq:L2bound_abs}, under the condition \eqref{eq:threshold_abs}, follows.

We now prove \eqref{eq:STP1Ab} by using a standard duality argument. 
By the definition of $\widetilde{\cR}$ \eqref{eq:tildeR} and Galerkin orthogonality \eqref{eq:epGogAb}, 
\begin{align*}
\big\|S^*(I-\widetilde{\Pi}_h)v\big\|^2_{\mathcal{H}_0}= \big\langle SS^*(I-\widetilde{\Pi}_h)v,  (I-\widetilde{\Pi}_h)v\big\rangle_{\mathcal{H}_0}  &= \widetilde{a} \big( \widetilde{\cR} SS^* (I-\widetilde{\Pi}_h)v-w_h,(I-\widetilde{\Pi}_h)v\big)\\
&\hspace{-3ex}= \widetilde{a} \big((I-\Pi_h) \widetilde{\cR} SS^* (I-\widetilde{\Pi}_h)v,(I-\widetilde{\Pi}_h)v\big).
\end{align*}
Then, by continuity of $\widetilde{a}$ and the bounds \eqref{eq:Rtildereg} and \eqref{eq:Sreg},
\begin{align*}
\big\|S^*(I-\widetilde{\Pi}_h)v\big\|^2_{\mathcal{H}_0} &\leq C\big\|(I-\Pi_h)\widetilde{\cR} SS^* (I-\widetilde{\Pi}_h)v\big\|_{\cH} \big\|(I-\widetilde{\Pi}_h)v\big\|_{\cH}\\
&\leq \|I-\Pi_h\|_{\Zspace_{\ell+1}\to \mathcal{H}}\big\|\widetilde{\cR} SS^* (I-\widetilde{\Pi}_h)v\big\|_{\Zspace_{\ell+1}} \big\|(I-\widetilde{\Pi}_h)v\big\|_{\cH},\\
&\leq C\|I-\Pi_h\|_{\Zspace_{\ell+1}\to \mathcal{H}}\big\|SS^* (I-\widetilde{\Pi}_h)v\big\|_{\Zspace_{\ell-1}} \big\|(I-\widetilde{\Pi}_h)v\big\|_{\cH},\\
&\leq C\|I-\Pi_h\|_{\Zspace_{\ell+1}\to \mathcal{H}}\big\|S^* (I-\widetilde{\Pi}_h)v\big\|_{\mathcal{H}_0} \big\|(I-\widetilde{\Pi}_h)v\big\|_{\cH}
\end{align*}
which implies the bound \eqref{eq:STP1Ab}, and hence \eqref{eq:L2bound_abs}.

Finally, we prove \eqref{eq:rel_error_abs}. By~\eqref{eq:vp_abs}, \eqref{eq:oscil}, and the abstract elliptic-regularity assumption \eqref{eq:ellipticRegA}, $u\in \Zspace_{\ell +1}$ with 
\begin{align*}
\N{u}_{\Zspace_{\ell+1}}\leq C\big(\N{u}_{\cH_0}+\N{g}_{\Zspace_{\ell-1}}\big)\leq C\big(\N{u}_{\cH_0}+\N{g}_{\cH^*}\big).
\end{align*}
The variational problem \eqref{eq:vp_abs} implies that 
\beqs
\N{g}_{\cH^*} = \sup_{v\in \cH, v\neq 0}\frac{|a(u,v)|}{\N{v}_{\cH}}\leq C \N{u}_{\cH},
\eeqs
and thus $\N{u}_{\Zspace_{\ell+1}} \leq C  \N{u}_{\cH}$.
 The bound \eqref{eq:H1bound_abs} then implies that 
\begin{align*}
\N{u-u_h}_{\mathcal{H}} \leq C_2\big(1+\eta(\cH_h)\big) \N{I- \Pi_h}_{\Zspace_{\ell+1}\to \cH}\N{u}_{\Zspace_{\ell+1}}
\end{align*}
and \eqref{eq:rel_error_abs} follows.

\subsection{Proof of Theorem \ref{thm:eta_abs}}\label{sec:eta_abs}

Let 
$\widetilde{\mathcal{R}}^*:\mathcal{H}^*\to \mathcal{H}$ be defined by 
\beq\label{eq:tildeR*}
\widetilde{a}(u,\widetilde{\mathcal{R}}^*f)=\langle u,f\rangle\quad\tfa u\in \mathcal{H},\,f\in\mathcal{H}^*
\eeq
(compare to \eqref{eq:tildeR});
i.e., $\widetilde{\mathcal{R}}^*$ is the adjoint solution operator for the sesquilinear form $\widetilde{a}$.
Repeating the proof of \eqref{eq:Rtildereg} with the elliptic-regularity assumption \eqref{eq:ellipticRegA} replaced by \eqref{eq:ellipticRegAadj}, we see that 
\beq\label{eq:Rtildereg*}
\|\widetilde{\mathcal{R}}^*\|_{\Zspace_{j-2}\to \Zspace_j}\leq C,\qquad 2\leq j\leq \ell+1.
\eeq

We now claim that 
\beq\label{eq:niceFormula}
\mathcal{R}^* = \widetilde{\mathcal{R}}^*(I + S \mathcal{R}^*). 
\eeq
Indeed, by the definitions of $\widetilde{a}$ \eqref{eq:tAab} and $\mathcal{R}^*$ \eqref{eq:R*_abs} 
the fact that $S=S^*$, and the definition of $\widetilde{\mathcal{R}}^*$ \eqref{eq:tildeR*}
\begin{align*}
\widetilde{a}(u, \mathcal{R}^*f) = a(u, \mathcal{R}^*f) + \langle S u, \mathcal{R}^* f\rangle 
= \langle u, f\rangle + \langle u, S\mathcal{R}^* f \rangle
&= \langle u, (I+S\mathcal{R}^*) f \rangle\\
&= \widetilde{a} \big( u, \widetilde{\mathcal{R}}^* (I + S \mathcal{R}^* )f\big)
\end{align*}
for all $u\in \cH$ and $f\in \cH^*$. The expression \eqref{eq:niceFormula} then follows from coercivity of $\widetilde{a}$ \eqref{eq:aTildeAb}.
Then, by \eqref{eq:niceFormula} and the mapping properties \eqref{eq:Rtildereg*} and \eqref{eq:Sreg} of $\widetilde{\mathcal R}^*$ and $S$, respectively, 
\begin{align*}
&\|(I-\Pi_h)\mathcal{R}^*g\|_{\mathcal{H}} 
\leq C\Big(
\|(I-\Pi_h)\widetilde{\mathcal{R}}^*g\|_{\mathcal{H}} 
+\|(I-\Pi_h)\widetilde{\mathcal{R}}^*S \mathcal{R}^*g\|_{\mathcal{H}} 
\Big)
 \\
&\quad\qquad
\leq 
C\Big(
\|(I-\Pi_h)\|_{ \Zspace_2\to \cH} \|\widetilde{\mathcal{R}}^*\|_{\cH_0\to \Zspace_2}\|g\|_{\cH_0} \\
&\quad\qquad\qquad\quad
+ \|(I-\Pi_h)\|_{\Zspace_{\ell+1}\to \cH}
\|\widetilde{\mathcal{R}}^*\|_{\Zspace_{\ell-1}\to \Zspace_{\ell+1}}
\|S\|_{\cH_0\to\Zspace_{\ell-1}}
 \|\mathcal{R}^*\|_{\cH_0\to \cH_0} \|g\|_{\cH_0} 
\Big)
\\
&\quad\qquad
\leq C\Big(\|(I-\Pi_h)\|_{ \Zspace_2\to \cH} \|g\|_{\cH_0} 
+ \|(I-\Pi_h)\|_{\Zspace_{\ell+1}\to \cH} 
\|\mathcal{R}^*\|_{\cH_0\to \cH_0} \|g\|_{\cH_0}\Big);
\end{align*}
the result \eqref{eq:eta_bound_abs} then follows from the definition of $\eta(\cH_h)$ \eqref{eq:eta_def_abs}.

\bre[The splitting \eqref{eq:niceFormula}]
Recalling the mapping properties of $\widetilde{\mathcal R}^*$ \eqref{eq:Rtildereg*} and  $S$ \eqref{eq:Sreg}, we see that 
\eqref{eq:niceFormula} splits the Helmholtz adjoint solution into a part with finite regularity but norm bounded independently of $\|\mathcal{R}^*\|_{\cH_0\to\cH_0}$ -- namely $\widetilde{\mathcal{R}}^*$ -- and a part with the highest possible regularity allowed by the coefficients and the domain (via  Assumptions \ref{ass:er} and \ref{ass:er2}) and norm bounded by 
$\|\mathcal{R}^*\|_{\cH_0\to\cH_0}$
-- namely $\widetilde{\mathcal{R}}^*S \mathcal{R}^*$. 

We highlight that such a splitting was previously achieved in the following papers.
\bit
\item  \cite{MeSa:10, MeSa:11, EsMe:12, LSW3, LSW4, GLSW1, BeChMe:22}, which split the solution into ``high-'' and ``low-'' frequency components; in these papers the scatterer (either an impenetrable obstacle or variable coefficients) is assumed to be analytic, and thus the ``low''-frequency part of the solution is analytic (at least near the scatterer).
\item \cite{ChNi:20}, which worked under elliptic-regularity assumptions analogous to Assumptions \ref{ass:er} and \ref{ass:er2}, and 
expanded the Helmholtz solution in a series whose terms increase with regularity, with the remainder have the highest possible regularity and being norm-bounded by $\|\mathcal{R}^*\|_{\cH_0\to\cH_0}$.
\eit
\ere

\section{Elliptic-regularity results}

This section collects the elliptic-regularity results that are used to verify that Assumption \ref{ass:er} holds for Helmholtz problems with truncation of the exterior domain either by a radial  PML (in \S\ref{sec:PML}) or an impedance boundary condition/exact Dirichlet-to-Neumann map (in \S\ref{sec:impedance}).
Let 
\beq\label{eq:L}
\cL u = -k^{-2} \nabla\cdot (A \nabla u ) - c^{-2} u,
\eeq
with associated sesquilinear form 
\beqs
a(u,v) = \int_\Omega \Big( k^{-2} (A\nabla u )\cdot\overline{\nabla v} - c^{-2} u \, \overline{v}\Big),
\eeqs
where $\Omega$ be a bounded Lipschitz domain with outward-pointing unit normal vector $n$.
The conormal derivative $\partial_{n,A}u$ is defined for $u\in H^2(\Omega)$ by 
$\partial_{n,A}u := n \cdot (A\nabla u)$; 
recall that $\partial_{n,A} u$ can be defined for $u\in H^1(\Omega)$ with $\cL u\in L^2(\Omega)$ by Green's identity; see, e.g., \cite[Lemma 4.3]{Mc:00}.

Let $\|\cdot\|_{H^1_k}$ be defined by \eqref{eq:1knorm}
and define higher-order weighted Sobolev norms by
\beq\label{eq:weighted_norms}
\N{v}^2_{H^m_k(\domaingen)}:= \sum_{0\leq |\alpha|\leq m}k^{-2|\alpha|}\N{\partial^\alpha v}^2_{L^2(\domaingen)}.
\eeq
The rationale for using these norms is that if a function $v$ oscillates with frequency $k$, then $|(k^{-1}\partial)^\alpha v|\sim  |v|$ for all $\alpha$; this is true, e.g., if $v(\bx) = \exp(\ri k \bx\cdot\ba)$. 
We highlight that many papers on the FEM applied to the Helmholtz equation use the weighted $H^1$ norm $\vertiii{v}^2:=\N{\nabla v}^2_{L^2(\domaingen)} + k^2\N{v}^2_{L^2(\domaingen)}$; we work with \eqref{eq:1knorm}/\eqref{eq:weighted_norms} instead, because weighting the $j$th derivative with $k^{-j}$ is easier to keep track of than weighting the $j$th derivative with $k^{-j+1}$.

\begin{assumption}\label{ass:McLean}
For all $x\in \Omega$, 
$A_{j\ell}(x)= A_{\ell j}(x)$ and 
\beqs
\Re \big( A(x) \xi, \xi\big)_2=\Re\sum_{j=1}^d\sum_{\ell=1}^d A_{j\ell}(x) \xi_\ell \overline{\xi_j} \geq c|\xi|^2 \quad\tfa \xi \in \Com^d.
\eeqs
\end{assumption}

\begin{theorem}[Local elliptic regularity near a Dirichlet or Neumann boundary]\label{thm:erDN}
Let $\Omega$ be a Lipschitz domain and let $G_1, G_2$ be open subsets of $\Rea^d$ with $G_1 \Subset G_2$ and $G_1 \cap \partial \Omega \neq \emptyset$. Let 
\beq\label{eq:Omegaj}
\Omega_j := G_j \cap \Omega, \,\,j=1,2, \quad\tand \quad \Gamma_2:= G_2 \cap \partial\Omega.
\eeq
Suppose that $A$ satisfies Assumption \ref{ass:McLean}, $A, c \in C^{m,1}(\overline{\Omega_2})$, and $\Gamma_2 \in C^{m+1,1}$ for some $m\in \mathbb{N}$.
Given $k_0>0$, there exists $C>0$ such that if $k\geq k_0$, 
 $u\in H^1(\Omega_2)$, $\cL u \in H^m(\Omega_2)$, and either $u=0$ or $\partial_{n,A}u=0$ on $\Gamma_2$, then 
\beq\label{eq:erDN}
\N{u}_{H^{m+2}_k(\Omega_1)} \leq C \big( \N{u}_{H^1_k(\Omega_2)} + \N{\cL u}_{H^m_k(\Omega_2)}\big).
\eeq
\end{theorem}

\bpf
In unweighted norms, this follows from, e.g., \cite[Theorems 4.7 and 4.16]{Mc:00}; the proof in the weighted norms \eqref{eq:weighted_norms} is very similar.
\epf

\begin{theorem}[Local elliptic regularity for the transmission problem]\label{thm:erT}
Let $\Omega_{\rm in}$ be a Lipschitz domain, and let $\Omega_{\rm out}:= \Rea^d \setminus \overline{\Omega_{\rm in}}$. 
Let $G_1, G_2$ be open subsets of $\Rea^d$ with $G_1 \Subset G_2$ and $G_1 \cap \partial \Omega_{\rm in} \neq \emptyset$. Let 
\beqs
\Omega_{\rm in/out, j}:= G_j \cap \Omega_{\rm in/out},\quad j=1,2, \quad\tand \Gamma_2 := G_{2} \cap \partial \Omega_{\rm in}.
\eeqs
Suppose that $A$ satisfies Assumption \ref{ass:McLean}, $A|_{\Omega_{\rm in/out, 2}}, c|_{{\Omega_{\rm in/out, 2}}}\in C^{m,1}(\overline{\Omega_{\rm in/out, 2})}$, and $\Gamma_2 \in C^{m+1,1}$ for some $m\in \mathbb{N}$.
Given $k_0>0$, there exists $C>0$ such that if $k\geq k_0$, $u_{\rm in/out}\in H^1(\Omega_{\rm in/out})$, $\cL u \in H^m(\Omega_{\rm in/out, 2})$, and
$\uin =\uout$ and $\partial_{n,A}u_{\rm in} =\zeta \partial_{n,A} u_{\rm out}$ on $\Gamma_2$ 
for some $\zeta>0$, then 
\begin{align}\nonumber
&\N{u_{\rm in}}_{H^{m+2}_k(\Omega_{\rm in, 1})}
+\N{u_{\rm out}}_{H^{m+2}_k(\Omega_{\rm out, 1})}\\
&\quad \leq C \Big( \N{\uin}_{H^1_k(\Omega_{\rm in, 2})}
+  \N{\uout}_{H^1_k(\Omega_{\rm out, 2})}
 + \N{\cL\uin}_{H^m_k(\Omega_{\rm in, 2})}
  + \N{\cL\uout}_{H^m_k(\Omega_{\rm out, 2})}\Big).\label{eq:erT}
\end{align}
\end{theorem}
\bpf
In unweighted norms, this is, e.g.,
\cite[Theorem 5.2.1(i)]{CoDaNi:10}  (and  \cite[Theorems 4.7 and 4.20]{Mc:00} when $\zeta=1$); the proof in the weighted norms \eqref{eq:weighted_norms} is very similar.
\epf

\begin{theorem}[Local elliptic regularity for the impedance problem]\label{thm:erimp}
Let $\Omega$ be a Lipschitz domain and let $G_1, G_2$ be open subsets of $\Rea^d$ with $G_1 \Subset G_2$ and $G_1 \cap \partial \Omega \neq \emptyset$. Let $\Omega_j$ and $\Gamma_2$ be defined by \eqref{eq:Omegaj}.
Suppose that, for some $m\in \mathbb{N}$, $\Gamma_2 \in C^{m+1,1}$. Given $k_0>0$, there exists $C>0$ such that if $k\geq k_0$, $u\in H^1(\Omega_2)$, $\Delta u \in H^m(\Omega_2)$,
and $(k^{-1}\partial_n -\ri)u=0$ on $\Gamma_2$, then 
\beq\label{eq:erimp}
\N{u}_{H^{m+2}_k(\Omega_1)} \leq C \big( \N{u}_{H^1_k(\Omega_2)} + \N{k^{-2}\Delta u}_{H^m_k(\Omega_2)}\big).
\eeq
\end{theorem} 

\begin{proof}
When $m=0$, the result can be obtained from 
\cite[Corollary 4.2/Theorem 4.3]{ChNiTo:20} by multiplying by $k^{-2}$ to switch to weighted norms, and using that the trace operator has norm bounded by $Ck^{1/2}$ from $H^1_k$ to $L^2$ (which can be obtained from, e.g., \cite[Theorem 5.6.4]{Ne:01} since the weighted norms there are, up to a constant, the weighted norms \eqref{eq:1knorm}).

The proof that~\eqref{eq:erimp} follows for $m>0$ is then standard and can be found e.g. in~\cite[\S6.3.2, Theorem 5]{Ev:98}. We repeat it here in the context of impedance boundary conditions for completeness.

We now prove that if the bound holds for $m=q$, then it holds for $m=q+1$ (assuming the appropriate regularity of the coefficients and the domain). 
Without loss of generality, we can change coordinates and work with $U:=B(0,s)\cap \{x_d>0\}$ and $V:=B(0,t)\cap \{x_d>0\}$ for some $0<t<s$. In these coordinates
$$
\widetilde{\cL} u:=(-k^{-2}a^{ij}\partial_{x_i}\partial_{x_j}-k^{-2}(b^i\partial_{x_i}-c))u=f,\qquad (-k^{-1}\partial_{x_d}-\ri)u=0\text{ on }\{x_d=0\}\cap \overline{U}.
$$
Suppose that for some $q\geq 0$, for any $0<t<s$, 
\beq\label{eq:induction0}
\N{u}_{H_k^{q+2}(V)}\leq C_t\big(\N{u}_{L^2(U)}+\N{f}_{H_k^q(U)}\big).
\eeq
Now suppose that $f\in H_k^{q+1}(U)$ and $a,b,c\in C^{q+1,1}(\overline{U})$, 
and let $W:=B(0,r)\cap \{x_d>0\}$ with $t<r<s$. By \eqref{eq:induction0},
\beq\label{eq:induction1}
\N{u}_{H_k^{q+2}(W)}\leq C\big(\N{u}_{L^2(U)}+\N{f}_{H_k^q(U)}\big),
\eeq
and, by interior elliptic regularity, $u\in H^{q+3}_{\loc}(U)$. 

The next step is to bound tangential derivatives of $u$:~let $\alpha$ be a multiindex with $|\alpha|=q+1$ and $\alpha_d=0$ (so that $\partial_x^\alpha$ is a tangential derivative). Let 
\beqs
\widetilde{f} := \widetilde{\cL} \big(k^{-|\alpha|}\partial_x^\alpha u\big) \quad\text{ so that }\quad
\widetilde{f}  =  [\widetilde{\cL},k^{-|\alpha|}\partial_x^\alpha]u+ k^{-|\alpha|}\partial_x^\alpha f
\eeqs
(where $[A,B]:=AB-BA$). 
Therefore,
\beq\label{eq:tildef}
 \|\widetilde{f}\|_{L^2(W)}\leq 
 C\big(\N{u}_{H^{q+2}_{k}(W)}+\N{f}_{H_k^{q+1}(W)}\big)
\leq C\big(\N{u}_{L^2(U)}+\N{f}_{H_k^{q+1}(U)}\big).
 \eeq
where  to obtain the last inequality we have used  \eqref{eq:induction1} and the fact that the coefficients of $\widetilde{\cL}$ are $C^{q+1,1}(\overline{U})$.
 Furthermore $k^{-|\alpha|}\partial_x^\alpha u$ satisfies the impedance boundary condition, since
\beq\label{eq:commute}
(-k^{-1}\partial_{x_d}-\ri)k^{-|\alpha|}\partial_x^\alpha u|_{x_d=0}=k^{-|\alpha|}\partial_x^\alpha\big[(-k^{-1}\partial_{x_d}-\ri )u|_{x_d=0}\big]=0.
\eeq
Therefore, by the analogue of \eqref{eq:induction0} with $q=0$ and $U$ replaced by $W$, \eqref{eq:induction1}, and \eqref{eq:tildef},
\begin{align*}
\big\|k^{-|\alpha|}\partial_x^\alpha u\big\|_{H_k^2(V)}&\leq C\big(\big\|k^{-|\alpha|}\partial_x^\alpha u\big\|_{L^2(W)}
+\big\|\widetilde{f}\big\|_{L^2(W)}\big)\leq C\big(\N{u}_{L^2(U)}+\N{f}_{H_k^{q+1}(U)}\big).
\end{align*}
In summary, recalling the definition of $\alpha$, we have proved that
\begin{align}\nonumber
&\big\|k^{-|\beta|}\partial_x^\beta u\big\|_{L^2(V)}\leq C\big(\N{u}_{L^2(U)}+\N{f}_{H_k^{q+1}(U)}\big)
\text{ for all $|\beta|= q+3$ with $\beta_d\in\{0,1,2\}$.}\label{eq:beta_bound_basic}
\end{align}
To prove that the bound \eqref{eq:induction0} holds with $q$ replaced by $q+1$, i.e.,
$$
\N{u}_{H_k^{q+3}(V)}\leq C\big(\N{u}_{L^2(U)}+\|f\|_{H_k^{q+1}(U)}\big),
$$
it is sufficient to prove that 
\begin{align*}\nonumber
&\big\|k^{-|\beta|}\partial_x^\beta u\big\|_{L^2(V)}\leq C\big(\N{u}_{L^2(U)}+\N{f}_{H_k^{q+1}(U)}\big)\\
&\hspace{2cm}\text{for all $|\beta|= q+3$ with $\beta_d\in\{0,\dots, q+3\}$.}
\end{align*}
We therefore now prove by induction that if 
\beq\label{eq:beta_bound_1}
\big\|k^{-|\beta|}\partial_x^\beta u\big\|_{L^2(V)}\leq C\big(\N{u}_{L^2(U)}+\N{f}_{H_k^{q+1}(U)}\big)
\eeq
for any $|\beta|=q+3$ with $\beta_d\in\{0,\dots, j\}$ for some $j\in \{2,\dots, q+2\}$, then \eqref{eq:beta_bound_1}  holds 
for $|\beta|=q+3$ with $\beta_d=j+1$. Combined with \eqref{eq:beta_bound_basic}, this completes the proof.

We therefore assume that $|\beta|=q+3$ with $\beta_d=j+1$. 
Then, putting $\beta=\gamma+\delta$ with $\delta=(0,\dots,0,2)$ and $|\gamma|=q+1$ (so that $\gamma_d = j-1$), and using that $u\in H^{q+3}_{\loc}(U)$, we have
\beq\label{eq:piano}
k^{-|\gamma|}\partial^\gamma \widetilde{\cL} u=a^{dd}k^{-|\beta|}\partial ^\beta u+ Bu \quad \tin V,
\eeq
where
$$
Bu=\sum_{\substack{|\alpha|\leq q+3,\,\alpha_d\leq j}}b_\alpha k^{-|\alpha|}\partial_x^\alpha u
$$
for appropriate $b_\alpha$. By the induction hypothesis \eqref{eq:beta_bound_1},
\beqs
\N{Bu}_{L^2(V)}\leq C\big(\N{u}_{L^2(U)}+\N{f}_{H_k^{q+1}(U)}\big).
\eeqs
Dividing \eqref{eq:piano} by $a^{dd}$, taking the $L^2(V)$ norm, and using that 
$1/a^{dd}$ is bounded, we have 
$$
\|k^{-|\beta|}\partial^\beta u\|_{L^2(V)}\leq C\big(\N{u}_{L^2(U)}+\|f\|_{H_k^{q+1}(U)}\big);
$$
i.e., we have proved that \eqref{eq:beta_bound_1} holds for $|\beta|=q+3$ with $\beta_d=j+1$, and the proof is complete.
\end{proof}

Let 
$\DtN: H^{1/2}(\GR)\to H^{-1/2}(\GR)$ be the Dirichlet-to-Neumann map, $u\mapsto k^{-1} \partial_r u$, for the Helmholtz equation $(k^{-2}\Delta +1)u=0$ posed in the exterior of $B_R$ and satisfying the Sommerfeld radiation condition 
 \beq\label{eq:src}
k^{-1} \pdiff{u}{r}(x) - \ri  u(x) = o \Big(\frac{1}{r^{(d-1)/2}}\Big) \quad\tas r:= |x|\tendi, \text{ uniformly in $\widehat{x}:= x/r$}.
 \eeq
For explicit expressions for $\DtN$ in terms of spherical harmonics and Hankel and Bessel functions, see, 
e.g., \cite[Equations 3.7 and 3.10]{MeSa:10}. 

\begin{theorem}[Elliptic regularity for the Laplacian on a ball with boundary condition involving $\DtN$]\label{thm:erDtN}
Given $R, k_0>0$, there exists $C>0$ such that the following is true for all $k\geq k_0$. Let  $B_R:= \{ x : |x|< R\}$ and 
suppose that $u\in H^1(B_R)$, $\Delta u \in H^m(B_R)$,
and $(k^{-1}\partial_n - T)u=0$ on $\partial B_R$,
where $T$ is one of 
\beq\label{eq:choiceT}
\DtN, \quad \DtN^*, \quad \tand \quad (\DtN + \DtN^*)/2.
\eeq
 Then 
\beq\label{eq:erDtN}
\N{u}_{H^{m+2}_k(B_R)} \leq C \big( \N{u}_{L^2(B_R)} + \N{k^{-2}\Delta u}_{H^m_k(B_R)}\big).
\eeq
\end{theorem} 

\begin{proof}
When $m=0$ and $T=\DtN$, the bound \eqref{eq:erDtN} is contained in \cite[Lemma 6.4]{LSW2}. The only specific property of $\DtN$ used in the proof is that 
\beq\label{eq:Nedelec}
-\Re \langle \DtN \phi,\phi \rangle \geq 0 \quad \tfa \phi \in H^{1/2}(\partial B_R)
\eeq
(see \cite[Theorem 2.6.4]{Ne:01}, \cite[Lemma 2.1]{ChMo:08}); thus \eqref{eq:erDtN} holds for $m=0$ and the other two choices of $T$ in \eqref{eq:choiceT}. Furthermore, 
applying \eqref{eq:erDtN} with $m=0$ to $\varphi u$ with $\varphi \in C^\infty(\Rea^d)$, 
$\varphi \equiv 0$ on $B_{R_2}$, and  
$\varphi \equiv 1$ on $(B_{R_3})^c$, 
with  $R_1<R_2<R_3<R$, we find that 
\beq\label{eq:train1}
\|u\|_{H_k^2(B_R\setminus B_{R_3})}\leq C\Big(\|u\|_{L^2(B_R\setminus B_{R_2})}+\|k^{-2}\Delta u\|_{L^2(B_R\setminus B_{R_2})}+k^{-1}\|u\|_{H^1_k(B_R\setminus B_{R_2})}\Big). 
\eeq
Let $\widetilde{\varphi} \in C^\infty(\Rea^d)$ be such that $\widetilde{\varphi}\equiv 1$
on $(B_{R_2})^c$, $\widetilde{\varphi}\equiv 0$ on $B_{R_1}$, and $|\nabla\widetilde{\varphi}|^2/|\widetilde{\varphi}|$ is bounded (this last condition can be achieved by making $\widetilde{\varphi}$ vanish quadratically at $\partial B_{R_1}$).
Applying Green's identity to $u$ and $\widetilde{\varphi} u$, and using \eqref{eq:Nedelec}, we find that
\beqs
\int_{B_R} \widetilde{\varphi} |k^{-1}\nabla u|^2 \leq -\Re\bigg( \int_{B_R} \overline{u} (k^{-1}\nabla \widetilde{\varphi}) \cdot (k^{-1}\nabla u)
+\widetilde{\varphi} \overline{u} (k^{-2}\Delta u )
 \bigg).
\eeqs
Using the Cauchy--Schwarz inequality and the inequality $2ab \leq \epsilon a^2 + \epsilon^{-1}b^2$ for all $a,b,\epsilon>0$, we obtain that 
\beq\label{eq:train2}
\|k^{-1}\nabla u\|_{L^2(B_R\setminus B_{R_2})}\leq C\big(\|u\|_{L^2(B_R\setminus B_{R_1})}+\|k^{-2}\Delta u\|_{L^2(B_R\setminus B_{R_1})}\big).
\eeq
Combining \eqref{eq:train1} and \eqref{eq:train2}, we obtain that 
\begin{equation}
\label{e:local}
\|u\|_{H_k^2(B_R\setminus B_{R_3})}\leq C\big(\|u\|_{L^2(B_R\setminus B_{R_1})}+\|k^{-2}\Delta u\|_{L^2(B_R\setminus B_{R_1})}\big). 
\end{equation}

We now repeat the argument for increasing $m$ used in Theorem~\ref{thm:erimp} except that we work in an annulus around $\partial B_R$ in polar coordinates. Without loss of generality $R>2$, so that, for $u\in C_{\rm comp}^\infty (B(0,R+1)\setminus \overline{B(0,R-1)}$, 
$$
-k^{-2}\Delta u = \Big(-k^{-2}\partial_r^2 -\frac{d-1}{k r}k^{-1}\partial_r -r^{-2}k^{-2}\Delta_\omega\Big)u,
$$
where $ [R-1,R+1]\times S^{d-1}\ni (r,\omega)\mapsto r\omega\in \mathbb{R}^d$ and $\Delta_\omega$ denotes the Laplacian on $S^{d-1}$.  
Now, $\DtN$ and $\Delta_\omega|_{\partial B_R}$  commute:~this can be seen, e.g., from the expression for $\DtN$ in terms of spherical harmonics on $\partial B_R$. 

Let $\chi \in C_{\rm comp}^\infty(B(0,R+1)\setminus \overline{B(0,R-1)})$ with $\chi \equiv 1$ near $\partial B_R$ and put $v=\chi u$ so that 
$$
-\Delta v = [-\Delta, \chi]u-\chi \Delta u
$$
Then,
$$
(-\Delta_\omega)^j (-\Delta v)=(-\Delta_\omega)^j\big([-\Delta, \chi]u-\chi \Delta u\big), \qquad (k^{-1}\partial_r-T)(-\Delta_\omega)^j v|_{\partial B_R}=0.
$$
Since $-\Delta_\omega$ commutes with $-\Delta$, 
$$
-\Delta (-\Delta_\omega)^j v=(-\Delta_\omega)^j (-\Delta v)=(-\Delta_\omega)^j\big([-\Delta, \chi]u-\chi \Delta u\big).
$$
In particular, letting $R_1<R_3<R$,  such that $B_R\setminus B_{R_1} \subset \{\chi \equiv 1\}$,~\eqref{e:local} implies
\begin{align*}
\|k^{-2j}(-\Delta_\omega)^j v\|_{H_k^{2}(B_R\setminus B_{R_3})}&\leq C\|k^{-2j}(-\Delta_\omega)^ju\|_{L^2(B_R\setminus B_{R_1})}
\\
&\hspace{2cm}
+\|k^{-2j}(-\Delta_\omega)^j k^{-2}\Delta u\|_{L^2(B_R\setminus B_{R_1})}\\
&\leq C\|k^{-2j}(-\Delta_\omega)^ju\|_{L^2(B_R\setminus B_{R_1})}+C\|k^{-2}\Delta u\|_{H_k^{2j}(B_R)}.
\end{align*}
Using that 
$$C\|k^{-2j}(-\Delta_\omega)^ju\|_{L^2(B_R\setminus B_{R_1})}\leq C\|k^{-2(j-1)}(-\Delta_\omega )^{j-1}u\|_{H_k^2(B_R\setminus B_{R_1})}$$
 together with the $m=0$ case and induction on $j$, we obtain for any $R>R_4>R_3$ that
$$
\|k^{-2j}(-\Delta_\omega)^j v\|_{H_k^{2}(B_R\setminus B_{R_4})}\leq C\|u\|_{L^2(B_R)}+C\|k^{-2}\Delta u\|_{H_k^{2j}(B_R)}.
$$

Applying elliptic regularity of $(-\Delta_\omega)$ for each fixed $R_4\leq r\leq R$, we obtain for any $|\alpha|+\ell\leq 2j+2$ and $\ell\in \{0,1,2\}$ that
$$
\|k^{-|\alpha|-\ell}\partial_\omega^\alpha\partial_r^\ell v\|_{L^2(B_R\setminus B_{R_4})}\leq C\|u\|_{L^2(B_R)}+C\|k^{-2}\Delta u\|_{H_k^{2j}(B_R)}.
$$
We then proceed as in the proof of Theorem~\ref{thm:erimp} starting from~\eqref{eq:beta_bound_1} to complete the proof of~\eqref{eq:erDtN}.
\epf

\section{Theorem \ref{thm:ep_abs}  applied to the PML problem}\label{sec:PML}

\subsection{Definition of the PML problem}\label{sec:PML_def}

\paragraph{Obstacles and coefficients for Dirichlet/Neumann/penetrable obstacle problem}

Let $\Ot,\OI\subset B_{R_0}:= \{ x : |x| < R_0\}\subset \Rea^d$, $d=2,3$, be bounded open sets with Lipschitz boundaries, $\Gt$ and $\GI$, respectively, such that $\Gt\cap \GI=\emptyset$, and $\Rea^d\overline{\setminus \OI}$ is connected. Let $\Omegaout:=\Rea^d\overline{\setminus \OI\cup \Ot}$ and $\Omegain:=(\Rea^d\overline{\setminus\OI})\cap \Ot$.

Let  $\Ascatout \in C^{0,1} (\Omegaout , \Rea^{d\times d})$ 
and $\Ascatin \in C^{0,1}(\Omegain, \Rea^{d\times d})$ 
be symmetric positive definite, let $\cscatout \in L^\infty(\Omegaout;\Rea)$, $\cscatin \in L^\infty(\Omegain;\Rea)$ 
be strictly positive, and let $\Ascatout$ and $\cscatout$ be such that 
 there exists $R_{\rm scat}>R_0>0$ such that 
\beqs
\overline{\Omega_-} \cup {\rm supp}(I- \Ascatout) \cup {\rm supp}(1-\cscatout) \Subset B_{R_{\rm scat}}.
\eeqs

The obstacle $\Omega_-$ is the impenetrable obstacle, on which we impose either a zero Dirichlet or a zero Neumann condition, and the obstacle $\Omegain$ is the penetrable obstacle, across whose boundary we impose transmission conditions. 

For simplicity, we do not cover the case when 
$\Omega_-$ is disconnected,  with Dirichlet boundary conditions on some connected components and Neumann boundary conditions on others, but the main results hold for this problem too (at the cost of introducing more notation).

\paragraph{Definition of the radial PML}

Let $\Rtr >\RPMLo>R_{\rm scat}$ and let $\Omega_{\tr}\subset \mathbb{R}^d$ be a bounded Lipschitz open set with $B_{\Rtr }\subset \Omega_{\tr} \subset  B_{C\Rtr }$ for some $C>0$ (i.e., $\Omega_{\tr}$ has characteristic length scale $\Rtr $).
Let $\Omega:=\Omega_{\tr}\cap (\Omegain \cup \Omegaout)$ and $\Gamma_{\tr}:=\partial\Omega_{\tr}$.
For $0\leq \theta<\pi/2$, let the PML scaling function $f_\theta\in C^{3}([0,\infty);\mathbb{R})$ be defined by $f_\theta(r):=f(r)\tan\theta$ for some $f$ satisfying
\begin{equation}
\label{e:fProp}
\begin{gathered}
\big\{f(r)=0\big\}=\big\{f'(r)=0\big\}=\big\{r\leq \RPMLo\big\},\quad f'(r)\geq 0,\quad f(r)\equiv r \text{ on }r\geq \RPMLt;
\end{gathered}
\end{equation}
i.e., the scaling ``turns on'' at $r=\RPMLo$, and is linear when $r\geq \RPMLt$. 
We note that $\Rtr $ can be $<\RPMLt$, i.e., we allow truncation before linear scaling is reached. 
Given $f_\theta(r)$, let 
\beq\label{eq:alpha_beta}
\alpha(r) := 1 + \ri f_\theta'(r) \quad \tand\quad \beta(r) := 1 + \ri f_\theta(r)/r.
\eeq
and let
\beq\label{eq:Ac}
A := 
\begin{cases}
\Ascatin 
\hspace{-1ex}
& \tin \Omegain,\\
\Ascatout 
\hspace{-1ex}
& \tin \Omegaout \cap B_{\RPMLo},\\
HDH^T 
\hspace{-1ex}
&\tin (B_{\RPMLo})^c
\end{cases}
\tand
\frac{1}{c^2} := 
\begin{cases}
\cscatin^{-2} 
\hspace{-1ex}
& \tin \Omegain,\\
\cscatout^{-2} 
\hspace{-1ex}
& \tin \Omegaout \cap B_{\RPMLo},\\
\alpha(r) \beta(r)^{d-1} 
\hspace{-1ex}
&\tin (B_{\RPMLo})^c,
\end{cases}
\eeq
where, in polar coordinates $(r,\varphi)$,
\beq\label{eq:DH2}
D =
\left(
\begin{array}{cc}
\beta(r)\alpha(r)^{-1} &0 \\
0 & \alpha(r) \beta(r)^{-1}
\end{array}
\right) 
\quad\tand\quad
H =
\left(
\begin{array}{cc}
\cos \varphi & - \sin\varphi \\
\sin \varphi & \cos\varphi
\end{array}
\right) 
\tfor d=2,
\eeq
and, in spherical polar coordinates $(r,\varphi, \phi)$,
\beq\label{eq:DH3}
D =
\left(
\begin{array}{ccc}
\beta(r)^2\alpha(r)^{-1} &0 &0\\
0 & \alpha(r) &0 \\
0 & 0 &\alpha(r)
\end{array}
\right) 
\tand
H =
\left(
\begin{array}{ccc}
\sin \varphi \cos\phi & \cos \varphi \cos\phi & - \sin \phi \\
\sin \varphi \sin\phi & \cos \varphi \sin\phi & \cos \phi \\
\cos \varphi & - \sin \varphi & 0 
\end{array}
\right) 
\eeq
for $d=3$ 
(observe that then $\Ascatout=I$ and $\cscatout^{-2}=1$ when $r=\RPMLo$ and thus $A$ and $c^{-2}$ are continuous at $r=\RPMLo$).

We highlight that, in other papers on PMLs, the scaled variable, which in our case is $r+\ri f_\theta(r)$, is often written as $r(1+ \ri \widetilde{\sigma}(r))$ with $\widetilde{\sigma}(r)= \sigma_0$ for $r$ sufficiently large; see, e.g., \cite[\S4]{HoScZs:03}, \cite[\S2]{BrPa:07}. Therefore, to convert from our notation, set $\widetilde{\sigma}(r)= f_\theta(r)/r$ and $\sigma_0= \tan\theta$.

Let 
\beq\label{eq:cH_PML}
\cH:=H^1_0(\Omega)  \quad \text{ or } \quad \{ v \in H^1(\Omega) \,:\, v=0 \text{ on } \Gamma_\tr\},
\eeq
with the former corresponding to zero Dirichlet boundary conditions on $\Omega_-$ and the latter corresponding to zero Neumann boundary conditions on $\Omega_-$.

\begin{definition}[A variational formulation of the PML problem]
Given $G\in \cH^*$ and $\zeta>0$, 
\beq\label{eq:PML_vf}
\text{ find } u \in \cH \,\tst\, a(u,v) = G(v) \,\tfa v \in \cH,
\eeq
where 
\beq\label{eq:PML_a}
a(u,v) := 
\left(\int_{\Omega\cap \Omegaout}  + \frac{1}{\zeta}\int_{\Omegain} \right)
\Big(k^{-2}(A\nabla u) \cdot\overline{\nabla v} - c^{-2} u \overline{v}\Big).
\eeq
\end{definition}

When 
\beq\label{eq:Gg}
G(v):= 
\left(\int_{B_{\RPMLo}\cap \Omegaout}  + \frac{1}{\zeta}\int_{\Omegain} \right)
c^{-2}g \overline{v}
\eeq
for $g\in L^2(\Omega)$ with $\supp \, g\subset B_{\RPMLo}$, 
the variational problem \eqref{eq:PML_vf} is a weak form of the problem
\beq\label{eq:olddognewtricks}
\begin{aligned}
 k^{-2}\cscatout^{2}\nabla\cdot (\Ascatout\nabla \uout) + \uout &=- g \quad\tin \Omegaout,\\
 k^{-2}\cscatin^{2}\nabla\cdot (\Ascatin\nabla \uin) + \uin &=- g \quad\tin \Omegain,\\
&\hspace{-4cm}\uin= \uout \quad\tand\quad\partial_{n, \Ascatin} \uin = \zeta\partial_{n,\Ascatout} \uout \quad\ton \Gt, \\
&\hspace{-4cm}\text{ either } \quad\uin = 0 \quad\text{ or }\quad \partial_{n,\Ascatin}\uin=0 \quad\ton \GI, 
\end{aligned}
\eeq
and with the Sommerfeld radiation condition approximated by a radial PML
(\eqref{eq:PML_vf} is obtained by multiplying the PDEs above by $c^{-2}_{\rm in/out} \alpha \beta^{d-1}$ and integrating by parts).

Using the fact that the solution of the true scattering problem exists and is unique with $\Ascatout,\Ascatin,\cscatout,\cscatin, \Omega_-,$ and $\Omegain$ described above (see, e.g., the discussion and references in \cite[\S1]{GrPeSp:19}),
the solution of \eqref{eq:PML_vf} exists and is unique (i) for fixed $k$ and sufficiently large $\Rtr-R_1$ by \cite[Theorem 2.1]{LaSo:98}, \cite[Theorem A]{LaSo:01}, \cite[Theorem 5.8]{HoScZs:03} and (ii) for fixed $\Rtr>R_1$ and sufficiently large $k$ by \cite[Theorem 1.5]{GLS2}.

For the particular data $G$ \eqref{eq:Gg}, it is well-known that, for fixed $k$, the error $\|u-v\|_{H^1_k(B_{\RPMLo}\setminus \Omega)}$ decays exponentially in $R_{\rm tr}-\RPMLo$ and $\tan\theta$; see 
\cite[Theorem 2.1]{LaSo:98}, \cite[Theorem A]{LaSo:01}, \cite[Theorem 5.8]{HoScZs:03}.
It was recently proved in \cite[Theorems 1.2 and 1.5]{GLS2} that the error $\|u-v\|_{H^1_k(B_{\RPMLo}\setminus \Omega)}$ also decreases exponentially in $k$.

\subsection{Showing that the PML problem fits in the abstract framework used in Theorem \ref{thm:ep_abs}}\label{sec:show_PML}

Recall that $\cH$ is defined by \eqref{eq:cH_PML} and let $\cH_0= L^2(\Omega)$. 
We work with the norm $\|\cdot\|_{H^1_k(\Omega)}$ \eqref{eq:1knorm} on $\cH$.

We first check that the sesquilinear form $a$ \eqref{eq:PML_a} is continuous and satisfies a G\aa rding inequality, with constants uniform for $\epsilon \leq \theta \leq \pi/2-\epsilon$. 

\ble[Bounds on the coefficients $A$ and $c$]\label{lem:coeff_bounds}
Given $A$ and $c$ as in \eqref{eq:Ac}, a scaling function $f(r)$ satisfying \eqref{e:fProp}, and $\epsilon>0$ there exist $\Amax$ and $\cmin$ such that, for all $\epsilon \leq \theta \leq \pi/2-\epsilon$, $x \in \Omega$, and $\xi,\zeta \in \mathbb{C}^d$,
\beqs
|(A(x)\xi,\zeta)_2|\leq \Amax \|\xi\|_2 \|\zeta\|_2
\quad\tand\quad
\frac{1}{|c(x)|^2} \geq \frac{1}{\cmin^2}.
\eeqs
\ele

\bpf
This follows from the definitions of $A$ and $c$ in \eqref{eq:Ac}, the definitions of $\alpha$ and $\beta$ in \eqref{eq:alpha_beta}, and the fact that $f_\theta(r):= f(r)\tan\theta$.
\epf

Continuity of $a$ \eqref{eq:contAb} with 
$\Ccont := \max\{ \Amax, \cmin^{-2} \}$ then follows from 
the  Cauchy-Schwarz inequality and the definition of $\|\cdot\|_{H^1_k(\Omega)}$ \eqref{eq:1knorm}.

\begin{assumption}\label{ass:PML}
When $d=3$, $f_\theta(r)/r$ is nondecreasing.
\end{assumption}

Assumption \ref{ass:PML} is standard in the literature; e.g., in the alternative notation described above it is that 
$\widetilde{\sigma}$ is non-decreasing -- see \cite[\S2]{BrPa:07}.

\bre\label{rem:JeffPML}
As noted above, the variational problem \eqref{eq:PML_vf} is obtained by multiplying the PDEs in \eqref{eq:olddognewtricks} by $c^{-2}_{\rm in/out} \alpha \beta^{d-1}$ and integrating by parts (as in \cite[\S3]{CoMo:98a}). 
If one integrates by parts the PDEs directly (as in, e.g., \cite[Lemma 4.2 and Equation 4.8]{HoScZs:03}), the resulting sesquilinear form satisfies Assumption \ref{ass:er} after multiplication by $\re^{\ri \omega}$, for some suitable $\omega$ (see Remark \ref{rem:omega}), without the need for Assumption \ref{ass:PML}.
\ere

\begin{lemma}\label{lem:strong_elliptic}
Suppose that $f_\theta$ satisfies Assumption \ref{ass:PML}.
With $A$ defined by \eqref{eq:Ac}, given $\epsilon>0$
there exists $\Amin>0$ such that, 
for all $\epsilon \leq \theta\leq \pi/2-\epsilon$,
\beqs
\Re \big( A(x) \xi, \xi\big)_2 \geq \Amin \|\xi\|_2^2 \quad\tfa \xi \in \mathbb{C}^d \tand x \in \Omega.
\eeqs
\end{lemma}

\bpf[Reference for the proof]
See, e.g., \cite[Lemma 2.3]{GLSW1}.
\epf

\begin{corollary}
\label{cor:Garding}
If $f_\theta$ satisfies Assumption \ref{ass:PML} then
\beqs
\Re a(w,w) \geq \Amin \|w\|^2_{H^1_k(\Omega)} - \big(\Amin + c^{-2}_{\min}\big) \|w\|^2_{L^2(\Omega)} \quad\tfa w\in \cH.
\eeqs
\end{corollary}

Let $\cR: L^2(\Omega)\to \cH$ be defined by 
$a(\cR g, v) =
(g,v)_{L^2(\Omega )}$ for all $v\in \cH$;
i.e., $\cR$ is the solution operator of the PML problem. 
The definition of $a$ and the facts that (with the matrices $H$ and $D$ defined by \eqref{eq:DH2}, \eqref{eq:DH3}) $H$ is real and the matrix $D$ is diagonal (and hence symmetric) imply that 
$a(\overline{u},v) = a(\overline{v},u)$ for all $u,v\in \cH$, and thus $\cR g = \overline{\cR^* \overline{g}}$. We therefore let
\beq\label{eq:Csol}
\Csol:= \N{\cR}_{L^2(\Omega)\to L^2(\Omega)} = \N{\cR^*}_{L^2(\Omega)\to L^2(\Omega)}.
\eeq
We highlight that (i) $\Csol$ is bounded by the norm of the solution operator of the true scattering problem (i.e., with the Sommerfeld radiation condition) by \cite[Theorem 1.6]{GLS2}, 
(ii) $\Csol \sim k$ when the problem is nontrapping (with this the slowest-possible growth in $k$), and (iii) an advantage of working with the weighted norms \eqref{eq:weighted_norms} is that $\Csol$ in fact describes the $k$-dependence of the Helmholtz solution operator between  $H^m_k$ and $H_k^{m+2}$ for any $m$.

\ble[The PML problem satisfies Assumptions \ref{ass:er} and \ref{ass:er2}]
\label{lem:checkPML}
Suppose that
$f_\theta$ satisfies Assumption \ref{ass:PML} and, for some $\ell \in \mathbb{Z}^+$, $\Ascatout, \Ascatin, \cscatout, \cscatin$ are all $C^{\ell-1,1}$, $f_\theta $ is $C^{\ell, 1}$, and 
 $\Gt, \GI$, and $\Gamma_{\rm tr}$ are all $C^{\ell,1}$.
Let 
\beq\label{eq:cZj}
\Zspace_j = \big\{ v : \vout \in H^j(\Omega\cap \Omegaout), \,\vin\in H^j(\Omegain)\big\}\cap \cH
\eeq
with norm 
\beq\label{eq:normZj}
\N{v}_{\Zspace_j}^2:= \N{\vout}^2_{H^j_k(\Omegaout\cap\Omega)} + 
\N{\vin}^2_{H^j_k(\Omegain)}.
\eeq
where the ``out'' and ``in'' subscripts denote restriction to $\Omegaout\cap\Omega$ and $\Omegain$, respectively.

Then $a$ defined by \eqref{eq:PML_a} satisfies Assumptions \ref{ass:er} and \ref{ass:er2} and given $\epsilon>0$ and $k_0>0$ there exists $C>0$ such the bounds \eqref{eq:ellipticRegA}, \eqref{eq:ellipticReg}, and \eqref{eq:ellipticRegAadj} hold for all $k\geq k_0$ and $\epsilon\leq \theta\leq \pi/2-\epsilon$.
\ele

\bpf
With $\cL$ defined by \eqref{eq:L},
\beqs
\sup_{\substack{v\in \mathcal{H} \euanspace \|v\|_{(\Zspace_{j-2})^*}=1}}|a(u,v)|=\N{\cL u}_{\Zspace_{j-2}}
\,\,\tand\,\,
\sup_{\substack{u\in \mathcal{H} \euanspace \|u\|_{(\Zspace_{j-2})^*}=1}}|a(u,v)|=\N{\cL^* v}_{\Zspace_{j-2}}.
\eeqs
Assumption \ref{ass:McLean} is then satisfied for both $\cL$ and $\cL^*$ by the definition \eqref{eq:Ac} of $A$, Lemma \ref{lem:strong_elliptic}, and the fact that $A$ is symmetric.

The bounds \eqref{eq:ellipticRegA} and \eqref{eq:ellipticRegAadj} then hold by combining Theorem \ref{thm:erDN} (used near $\Gamma_-$ and $\Gamma_\tr$) and Theorem \ref{thm:erT} (used near $\Gt$) and using the fact that, by Green's identity, for $u\in H^1_0(\Omega)$ with $\cL u \in L^2(\Omega)$ and $\partial_{n, \Ascatin} \uin = \zeta \partial_{n,\Ascatout} \uout$ on $\partial\Omegain$, 
\begin{align*}
&\N{\uin}_{H^1_k(\Omegain)}
+\N{\uout}_{H^1_k(\Omegaout)}\\
&\qquad\leq C \Big( \N{\uin}_{L^2(\Omegain)} 
+\N{\uout}_{L^2(\Omegaout)} 
+ \N{\cL \uin}_{L^2(\Omegain)}
+ \N{\cL \uout}_{L^2(\Omegaout)}
\Big)
\end{align*}
(so that the $H^1_k$ norms on the right-hand sides of \eqref{eq:erDN} and \eqref{eq:erT} can be replaced by $L^2$ norms).
Since the operator associated with the sesquilinear form $\Re a$ is 
\beqs
\left(\frac{\cL + \cL^*}{2}\right)u = 
 -k^{-2} \nabla\cdot \left(\frac{A+\overline{A}}{2} \nabla u \right) - \left(\frac{c^{-2}+ \overline{c}^{-2}}{2}\right) u
\eeqs
and the matrix $A$ is symmetric, 
this operator also satisfies Assumption \ref{ass:McLean}. The bound \eqref{eq:ellipticReg} then holds by a very similar argument.
\epf

\subsection{Theorem \ref{thm:ep_abs}  applied to the PML problem}

\begin{assumption}\label{ass:spaces}
Given $p\in \mathbb{Z}^+$, $(\fdspace)_{h>0}$ are such that the following holds. 
There exists $C>0$ such that, for all $h>0$, $0\leq j \leq m+1\leq p+1$, and 
 $v \in \Zspace_{m+1}$ defined by \eqref{eq:cZj}, 
 there exists $\mathcal{I}_{h,p}v \in \fdspace$ such that
\begin{align}\nonumber
&\big|\vout- (\mathcal{I}_{h,p}v)_{\rm out}\big|_{H^j(\Omegaout\cap\Omega )} 
+\big|\vin- (\mathcal{I}_{h,p}v)_{\rm in} \big|_{H^j(\Omegain )} \\
&\hspace{3cm}\leq C h^{m + 1-j} \big(\|\vout\|_{H^{m +1}(\Omegaout\cap\Omega )}+\|\vin\|_{H^{m +1}(\Omegain)}\big).
\label{eq:pp_approx}
\end{align}
where the ``out'' and ``in'' subscripts denote restriction to $\Omegaout\cap \Omega$ and $\Omegain$, respectively.
\end{assumption}

Assumption \ref{ass:spaces} holds when $(\fdspace)_{h>0}$ consists of piecewise degree-$p$ polynomials on shape-regular simplicial triangulations, indexed by the meshwidth; see, e.g.,
\cite[Theorem 17.1]{Ci:91}, \cite[Proposition 4.4.20]{BrSc:08}.

\begin{theorem}[Existence, uniqueness, and error bound in the preasymptotic regime for the PML problem]\label{thm:main}
Suppose that $\zeta>0$, $f_\theta$ satisfies Assumption \ref{ass:PML}, and, for some $\ell \in \mathbb{Z}^+$, $\Ascatout, \Ascatin, \cscatout, \cscatin$ are all $ C^{\ell-1,1}$,  
$f_\theta$ is $C^{\ell, 1}$, 
and  $\Gt, \GI,$ and $\Gamma_{\tr}$ are all $C^{\ell,1}$. 
Let $\Csol$ be defined by \eqref{eq:Csol}, and assume that $\{\cH_h\}_{h>0}$ satisfy Assumption \ref{ass:spaces}.
 Given $\epsilon>0$ and $p\in \mathbb{Z}^+$ with $p\geq \ell$, there exists $k_0>0$ and $C_j, j=1,2,3$, such that the following is true for all $k\geq k_0$ and $\e\leq \theta\leq \pi/2-\e$. 
 
The solution $u$ of the PML problem \eqref{eq:PML_vf} exists and is unique, and if 
\beq\label{eq:threshold_old}
(hk)^{2\ell} \Csol 
\leq C_1
\eeq
then the Galerkin solution $u_h$, exists, is unique, and satisfies
\begin{gather}
\label{eq:H1bound_old}
\N{u-u_h}_{H^1_k(\Omega)}\leq C_2 \Big(1 + (hk)^\ell \Csol\Big) \min_{v_h \in \fdspace} \N{u-v_h}_{H^1_k(\Omega)},\\
\label{eq:L2bound_old}
\N{u-u_h}_{L^2(\Omega)}\leq C_3 \Big( hk + (hk)^\ell \Csol\Big) \min_{v_h \in \fdspace} \N{u-v_h}_{H^1_k(\Omega)}.
\end{gather}
If, in addition, $g\in H^{\ell-1}(\Omega)\cap \cH$ (with $\cH$ defined by \eqref{eq:cH_PML}) with
\beq\label{eq:oscil_old}
\N{g}_{H^{\ell-1}_k(\Omega)}\leq C \N{g}_{\cH^*}
\eeq
for some $C>0$, then there exists $C_4>0$ such that if $h$ satisfies \eqref{eq:threshold_old} then
\beq\label{eq:rel_error_old}
\frac{\N{u-u_h}_{H^1_k(\Omega)} 
}{
\N{u}_{H^1_k(\Omega)}
}
\leq C_4 
 \Big( 1+  (hk)^\ell \Csol\Big)(hk)^\ell.
\eeq
\end{theorem}

When $p=\ell$, i.e., the polynomial degree matches the regularity of the domain and coefficients, 
 \eqref{eq:threshold_old} becomes the condition \eqref{eq:threshold}, and the bounds \eqref{eq:H1bound_old}, \eqref{eq:L2bound_old}, and \eqref{eq:rel_error_old} become \eqref{eq:H1bound}, \eqref{eq:L2bound}, and \eqref{eq:rel_error}, respectively.

\bpf[Proof of Theorem \ref{thm:main}]
By the results in \S\ref{sec:show_PML}, $a$ defined by \eqref{eq:PML_a} satisfies the assumptions of Theorems \ref{thm:ep_abs} and \ref{thm:eta_abs} with $\Ccont, \CGo, \CGt, \Cell$, and $\Celladj$ all independent of $k$. 
By \eqref{eq:pp_approx}, the definition of $\|\cdot\|_{\Zspace_j}$ \eqref{eq:normZj}, and the definition \eqref{eq:weighted_norms} of the weighted norms, $\|I-\Pi_h\|_{\Zspace_{m+1} \to \cH}\leq C(hk)^m$.
This bound along with Theorem \ref{thm:eta_abs} and \eqref{eq:Csol} imply the bound 
$\eta(\fdspace)\leq C ( hk + (hk)^\ell\Csol),$ and 
the result then follows from Theorem \ref{thm:ep_abs}, using that $hk \leq C$ when \eqref{eq:threshold_old} holds.
\epf

\section{Theorem \ref{thm:ep_abs} applied to the exact DtN/impedance problems}\label{sec:impedance}

\subsection{Definition of the exact DtN/impedance problems}

Let $\Ascatout, \Ascatin, \cscatout, \cscatin, \Omega_-, \Omegain,$ and $\Omega_\tr$ be as in \S\ref{sec:PML_def}. Let
\beqs
A := 
\begin{cases}
\Ascatin 
\hspace{-1ex}
& \tin \Omegain,\\
\Ascatout 
\hspace{-1ex}
& \tin \Omegaout \cap \Omega,
\end{cases}
\quad\tand\quad
\frac{1}{c^2} := 
\begin{cases}
\cscatin^{-2} 
\hspace{-1ex}
& \tin \Omegain,\\
\cscatout^{-2} 
\hspace{-1ex}
& \tin \Omegaout\cap\Omega.
\end{cases}
\eeqs
Let 
\beq\label{eq:cH_imp}
\cH:=\{ v \in H^1(\Omega) \,:\, v=0 \text{ on } \partial\Omega_-\}  \quad \text{ or } \quad H^1(\Omega),
\eeq
with the former corresponding to zero Dirichlet boundary conditions on $\Omega_-$ and the latter corresponding to zero Neumann boundary conditions on $\Omega_-$.

\begin{definition}[Variational formulation of the impedance/exact DtN problems]
Let \emph{either} $T=\ri I$ (with no further constraint on $\Omega_{\tr}$) \emph{or} $T=\DtN$ with $\Omega_{\tr}= B(0,R_{\tr})$. 
Given $G\in \cH^*$ and $\zeta>0$, 
\beq\label{eq:imp_vf}
\text{ find } u \in \cH \,\tst\, a(u,v) = G(v) \,\tfa v \in \cH,
\eeq
where 
\beq\label{eq:imp_a}
a(u,v) := 
\left(\int_{\Omega\cap \Omegaout}  + \frac{1}{\zeta}\int_{\Omega\cap\Omegain} \right)
\Big(k^{-2}(A\nabla u) \cdot\overline{\nabla v} - c^{-2} u \overline{v}\Big) - 
k^{-1} \langle T u, v \rangle_{\Gamma_\tr}.
\eeq
\end{definition}
The solution of this variational problem exists and is unique; see, e.g., \cite[Theorem 2.4]{GrSa:20} and the discussion and references in \cite[\S1]{GrPeSp:19}.

\subsection{Showing that the exact DtN/impedance problems fit in the abstract framework used in Theorem \ref{thm:ep_abs}}\label{sec:show_imp}

The proofs that the sesquilinear form $a$ is continuous and satisfies a G\aa rding inequality are very similar to those for the PML problem in \S\ref{sec:show_PML} (in fact, they are simpler because there is no PML scaling parameter in which the bounds need to be uniform). When $T=\DtN$, the proof of the G\aa rding inequality uses \eqref{eq:Nedelec} and 
the proof of continuity uses that $|\langle \DtN u, v\rangle_{\partial B_{\Rtr}}|\leq C k \|u \|_{H^1_k(\Omega)} \|v\|_{H^1_k(\Omega)}$ \cite[Equation 3.4a]{MeSa:10}.

\ble[The exact DtN/impedance problems satisfy Assumptions \ref{ass:er} and \ref{ass:er2}]
\label{lem:checkimp}
Suppose that, for some $\ell \in \mathbb{Z}^+$, $\Ascatout, \Ascatin, \cscatout, \cscatin$ are all $C^{\ell-1,1}$  
and  $\Gt$, $\GI$, and $\Gamma_{\rm tr}$ are all $C^{\ell,1}$. 
With $\Zspace_j$ and its norm defined by \eqref{eq:cZj} and \eqref{eq:normZj}, $a$ defined by \eqref{eq:imp_a} satisfies Assumptions \ref{ass:er} and \ref{ass:er2} and given $k_0>0$ there exists $C>0$ such the bounds 
\eqref{eq:ellipticRegA}, \eqref{eq:ellipticReg}, and \eqref{eq:ellipticRegAadj}
hold for all $k\geq k_0$.
\ele

\bpf
This is very similar to the proof of Lemma \ref{lem:checkPML}.
For the impedance problem, 
the regularity assumptions \eqref{eq:ellipticRegA} and \eqref{eq:ellipticRegAadj} follow by combining Theorem \ref{thm:erDN} used near $\partial\Omega_-$, 
Theorem \ref{thm:erT} used near $\partial \Omegain$, and Theorem \ref{thm:erimp} used near $\Gamma_\tr$.
The regularity assumption \eqref{eq:ellipticReg} follows by 
combining Theorem \ref{thm:erDN} used near $\partial\Omega_-$, 
Theorem \ref{thm:erT} used near $\partial \Omegain$, and now Theorem \ref{thm:erDN} (with Neumann boundary condition) used near $\Gamma_\tr$. Indeed, 
near $\Gamma_\tr$, the operator associated with $(\Re a)$ is $-k^{-2}\Delta -1$ with Neumann boundary conditions (coming from $\Ascatout=I$ and $\cscatout=1$ near $\Gamma_\tr$ and the fact that no boundary condition is imposed on $\Gamma_\tr$ in $\cH$ \eqref{eq:cH_imp}).

For the exact DtN problem, let  $\chi \in C^\infty(\Rea^d)$ be such that $\chi \equiv 0$ on $B_{\Rscat}$ and $\chi \equiv 1$ on $(B_{(\Rscat+ \Rtr)/2})^c$.
Then apply Theorem \ref{thm:erDtN} to $\chi u$ and Theorems \ref{thm:erDN} and  \ref{thm:erT} to $(1-\chi)u$. Note that (i) in these applications, the terms arising from the commutator of $\cL$ and $\chi$ are lower-order in both $k$ and Sobolev index and (ii) the different choices of $T$ in \eqref{eq:choiceT} correspond to 
 \eqref{eq:ellipticRegA}, \eqref{eq:ellipticRegAadj}, and \eqref{eq:ellipticReg}, respectively.
\epf

\subsection{Theorem \ref{thm:ep_abs}  applied to the exact DtN/impedance problems}

\begin{theorem}[Existence, uniqueness, and error bound in the preasymptotic regime for the exact DtN/impedance problems]\label{thm:main2}
Suppose that $\zeta>0$ and, for some $\ell \in \mathbb{Z}^+$, $\Ascatout, \Ascatin, \cscatout, \cscatin$ are all $C^{\ell-1,1}$ 
and 
  $\Gt, \GI,$ and $\Gamma_{\tr}$ are all $C^{\ell,1}$. 
Let $\Csol$ be defined by \eqref{eq:Csol}, and assume that $\{\cH_h\}_{h>0}$ satisfy Assumption \ref{ass:spaces}.
Given $p\in \mathbb{Z}^+$ with $p\geq \ell$, there exists $k_0>0$ and $C_j, j=1,2,3$, such that the following is true for all $k\geq k_0$.
  
The solution $u$ of the exact DtN/impedance problem \eqref{eq:imp_vf} exists and is unique, and if \eqref{eq:threshold_old} holds 
then the Galerkin solution $u_h$, exists, is unique, and satisfies  the bounds \eqref{eq:H1bound_old} and \eqref{eq:L2bound_old}.
If, in addition, $g\in H^{\ell-1}(\Omega)\cap \cH$ (with $\cH$ defined by \eqref{eq:cH_imp}) with \eqref{eq:oscil_old}
for some $C>0$, then there exists $C_4>0$ such that if $h$ satisfies \eqref{eq:threshold_old} then the bound \eqref{eq:rel_error_old} holds.
\end{theorem}

Given Lemma \ref{lem:checkimp}, the proof of Theorem \ref{thm:main2} is very similar to the proof of Theorem \ref{thm:main}, and so we omit it for brevity.

\section*{Acknowledgements}

We thank the referees for their constructive comments. EAS was supported by EPSRC grant EP/R005591/1 and JG was supported by EPSRC grants EP/V001760/1 and EP/V051636/1.

\bibliographystyle{siamplain}
\bibliography{biblio_combined_sncwadditions}

\end{document}